\pdfoutput=1
\RequirePackage{ifpdf}
\ifpdf % We are running pdfTeX in pdf mode
\documentclass[pdftex]{sigma}
\else
\documentclass{sigma}
\fi

\def\N{\mathbb{N}}
\def\C{\mathbb{C}}
\def\a{\mathbf{a}}
\def\b{\mathbf{b}}
\def\G{\mathcal{G}}
\def\Be{\mathcal{B}}
\def\al{\boldsymbol{\alpha}}
\def\be{\boldsymbol{\beta}}

\def\Z{\mathbb{Z}}
\def\L{\mathcal{L}}
\def\A{\mathbf{A}}
\def\B{\mathbf{B}}
\DeclareMathOperator*{\res}{\mathrm{res}}

\numberwithin{equation}{section}

\newtheorem{Theorem}{Theorem}[section]
\newtheorem{Corollary}[Theorem]{Corollary}
{ \theoremstyle{definition}
\newtheorem{Remark}[Theorem]{Remark} }

\begin{document}

\allowdisplaybreaks

\renewcommand{\thefootnote}{$\star$}

\newcommand{\arXivNumber}{1602.07375}

\renewcommand{\PaperNumber}{052}

\FirstPageHeading

\ShortArticleName{Hypergeometric Dif\/ferential Equation and New Identities for the Coef\/f\/icients}

\ArticleName{Hypergeometric Dif\/ferential Equation\\ and New Identities for the Coef\/f\/icients\\ of N{\o}rlund and B\"{u}hring\footnote{This paper is a~contribution to the Special Issue on Orthogonal Polynomials, Special Functions and Applications.
The full collection is available at \href{http://www.emis.de/journals/SIGMA/OPSFA2015.html}{http://www.emis.de/journals/SIGMA/OPSFA2015.html}}}

\Author{Dmitrii KARP~$^{\dag\ddag}$ and Elena PRILEPKINA~$^{\dag\ddag}$}
\AuthorNameForHeading{D.~Karp and E.~Prilepkina}
\Address{$^\dag$~Far Eastern Federal University, 8~Sukhanova Str., Vladivostok, 690950, Russia}
\Address{$^\ddag$~Institute of Applied Mathematics, Far Eastern Branch of the Russian Academy of Sciences,\\
\hphantom{$^\ddag$}~7~Radio Str., Vladivostok, 690041, Russia}
\EmailD{\href{mailto:dimkrp@gmail.com}{dimkrp@gmail.com}, \href{mailto:pril-elena@yandex.ru}{pril-elena@yandex.ru}}
\URLaddressD{\url{http://dmkrp.wordpress.com}}

\ArticleDates{Received February 25, 2016, in f\/inal form May 15, 2016; Published online May 21, 2016}

\Abstract{The fundamental set of solutions of the generalized hypergeometric dif\/ferential equation in the neighborhood of unity has been built by N{\o}rlund in 1955. The behavior of the generalized hypergeometric function in the neighborhood of unity has been described in the beginning of 1990s by B\"{u}hring, Srivastava and Saigo. In the f\/irst part of this paper we review their results rewriting them in terms of Meijer's $G$-function and explaining the interconnections between them. In the second part we present new formulas and identities for the coef\/f\/icients that appear in the expansions of Meijer's $G$-function and generalized hypergeometric function around unity. Particular cases of these identities include known and new relations for Thomae's hypergeometric function and forgotten Hermite's identity for the sine function.}

\Keywords{generalized hypergeometric function; hypergeometric dif\/ferential equation; Meijer's $G$-function; Bernoulli polynomials; N{\o}rlund's coef\/f\/icients; B\"{u}hring's coef\/f\/icients}

\Classification{33C20; 33C60; 34M35}

\renewcommand{\thefootnote}{\arabic{footnote}}
\setcounter{footnote}{0}

\vspace{-2mm}

\section{Introduction}
We will use standard notation $\Z$, $\N$ and $\C$ to denote integer, natural and complex numbers, respectively; $\N_0=\N\cup\{0\}$. The hypergeometric dif\/ferential equation ($D=z\frac{d}{dz}$)
\begin{gather}\label{eq:hyper-equation}
\left\{(D-a_1)(D-a_2)\cdots(D-a_p)-z(D+1-b_1)(D+1-b_2)\cdots(D+1-b_p)\right\}y=0
\end{gather}
for $p=2$ was f\/irst considered by Euler and later studied by Gauss, Kummer, Riemann, Papperitz and Schwarz, among others, see \cite[Section~2.3]{AAR}. For general $p>2$ it was probably f\/irst investigated by Thomae~\cite{Thomae}. Note that our choice of parameters dif\/fers slightly from that of \cite[(16.8.3)]{NIST}. As will become apparent in the sequel, this choice is more convenient if the solution is to be built in terms of Meijer's $G$-functions and not generalized hypergeometric functions.
Equation~(\ref{eq:hyper-equation}) is of Fuchsian type and has three regular singularities located at the points~$0$,~$1$,~$\infty$.
 The local exponents read \cite[(2.6)--(2.8)]{BeukersHeckman}
\begin{alignat*}{3}
 & a_1, a_2,\ldots,a_p \quad && \text{at}~z=0, &
\\
 & 0, 1, 2, \ldots, p-2, \sum(b_i-a_i)-1 \quad && \text{at}~z=1, &
\\
 & 1-b_1,1-b_2,\ldots,1-b_p \quad && \text{at}~z=\infty.&
\end{alignat*}

The fundamental sets of solutions around the points $z=0$ and
$z=\infty$ were found by Thomae in \cite{Thomae} and are expressed
in terms of the generalized hypergeometric series
\begin{gather*}%\label{eq:pFqdefined}
{_{p}F_q}\left( \begin{matrix}\a\\ \b\end{matrix}\,\vline\,z \right)={_{p}F_q}\left(\a;\b;z\right)=
\sum\limits_{n=0}^{\infty}\frac{(a_1)_n(a_2)_n\cdots(a_{p})_n}{(b_1)_n(b_2)_n\cdots(b_q)_nn!}z^n,
\end{gather*}
where $(a)_n=\Gamma(a+n)/\Gamma(a)$ denotes the rising factorial
and $\a=(a_1,\ldots,a_p)$, $\b=(b_1,\ldots,b_q)$ are complex vectors such that $-b_i\notin\N_0$
for $i=1,\ldots,q$. If the components of $\a$ are distinct modulo~$\Z$,
then the basis of solutions near $z=0$ is given by
\cite[(1.13)]{Norlund}
\begin{gather*}
z^{a_k}{_{p}F_{p-1}}\left( \begin{matrix}1-\b+a_k\\ 1-\a_{[k]}+a_k\end{matrix}\,\vline\, z \right),\qquad k=1,\ldots,p,
\end{gather*}
where $\a_{[k]}$ signif\/ies the vector $\a$ with the element~$a_k$ omitted and $\a+\alpha$ is understood as $(a_1+\alpha$, $\ldots,a_p+\alpha)$. If the components of~$\b$ are distinct modulo $\Z$ then the fundamental system of solutions near $z=\infty$ is given by \cite[(1.14)]{Norlund}
\begin{gather*}
z^{b_k-1}{_{p}F_{p-1}}\left( \begin{matrix}1+\a-b_k\\1+\b_{[k]}-b_k\end{matrix}\,\vline\,\frac{1}{z} \right),\qquad k=1,\ldots,p.
\end{gather*}

Finally, the fundamental set of solutions around the point $z=1$ was found by the remarkable Danish mathematician Niels Erik
N{\o}rlund in his milestone work~\cite{Norlund}. These facts are well-known and have been frequently cited in the literature. It seems to be less known that the fundamental solutions constructed by N{\o}rlund can be expressed in terms of $G$-function introduced some 15~years earlier by Meijer, see~\cite{Meijer}. In fact, Meijer himself studied the hypergeometric
dif\/ferential equation more general than~(\ref{eq:hyper-equation}) and built the basis of solutions in the neighborhood of zero and inf\/inity in terms of $G$-functions. The connection between N{\o}rlund's solutions and Meijer's $G$-function was
observed by Marichev and Kalla~\cite{MK} and Marichev~\cite{Marichev} but these two papers remained largely unnoticed.
The fact that each solution around $z=1$ must equal to a linear combination of fundamental solutions around $z=0$ is ref\/lected for $p=2$ in the following identity due to Gauss \cite[Theorem~2.3.2]{AAR}
\begin{gather}
{_2F_1} \left(\begin{matrix}\alpha_1, \alpha_2\\\beta\end{matrix};1-z \right)=
\frac{\Gamma(\beta)\Gamma(\beta-\alpha_1-\alpha_2)}{\Gamma(\beta-\alpha_1)\Gamma(\beta-\alpha_2)}
{_2F_1} \left(\begin{matrix}\alpha_1,\alpha_2\\\alpha_1+\alpha_2-\beta+1\end{matrix};z \right)
\nonumber\\
\hphantom{{_2F_1} \left(\begin{matrix}\alpha_1, \alpha_2\\\beta\end{matrix};1-z \right)=}{}
+\frac{\Gamma(\beta)\Gamma(\alpha_1+\alpha_2-\beta)}{\Gamma(\alpha_1)\Gamma(\alpha_2)}z^{\beta-\alpha_1-\alpha_2}
{_2F_1} \left(\begin{matrix}\beta-\alpha_1,\beta-\alpha_2\\\beta-\alpha_1-\alpha_2+1\end{matrix};z \right).\label{eq:2F1connection}
\end{gather}
The fact that each solution around $z=0$ must equal to a linear combination of fundamental solutions around $z=1$ leads to
essentially the same identity with $z$ replaced by $1-z$. For $p>2$, however the above connection formula has \emph{two
different generalizations}. One of them is the expansion (\ref{eq:Gp0pp}) of~$G^{p,0}_{p,p}$ (a solution around $z=1$) into a sum of ${}_pF_{p-1}$ (fundamental solutions around $z=0$) which can also be obtained by applying the residue theorem to the def\/inition of $G$-function and is found in standard references on $G$-function, see, for instance, \cite[(16.17.2)]{NIST}. The other one is the expansion of ${}_pF_{p-1}$ into a sum of $G^{p,0}_{p,p}$ with $p-1$ instances of~$G^{2,p}_{p,p}$ as given by formula~(\ref{eq:Norlund5.40}) below. This expansion, discovered by N{\o}rlund without any mentioning of $G$-function has been reproduced in~\cite{MK} in a slightly dif\/ferent form but seems to be forgotten afterwards. A related expansion, again with no reference to $G$-function, has been then found in \cite{Buehring87,Buehring92} by B\"{u}hring whose goal was to describe the behavior of ${}_pF_{p-1}(z)$ near $z=1$. The logarithmic cases have been studied by Saigo and Srivastava in~\cite{SS}. Let us also mention that the monodromy group of the generalized hypergeometric dif\/ferential equation has been constructed by Beukers and Heckman in~\cite{BeukersHeckman}.

This paper is organized as follows. Section~\ref{section2} is of survey nature: we review the results of N{\o}rlund and B\"{u}hring and rewrite them in terms of $G$-function. We also reveal connections between N{\o}rlund's and B\"{u}hring's expansions and relate them to various results obtained in statistics literature. In Section~\ref{section3} we present some new formulas for N{\o}rlund's coef\/f\/icients and utilize the above mentioned connections to derive various new identities for N{\o}rlund's and B\"{u}hring's coef\/f\/icients which reduce to hypergeometric and trigonometric identities for small~$p$. The most striking of these identities,
\begin{gather}\label{eq:Ptol}
\sum\limits_{k=1}^{p}\frac{\prod\limits_{i=1}^{p}\sin(\beta_i-\alpha_k)}
{\prod\limits_{\substack{i=1\\i\ne{k}}}^{p}\sin(\alpha_i-\alpha_k)}
=\sin\left(\sum\limits_{k=1}^{p}(\beta_k-\alpha_k)\right),
\end{gather}
can be viewed as a generalization of Ptolemy's theorem: for a quadrilateral inscribed in a circle the product of the lengths of its diagonals is equal to the sum of the products of the lengths of the pairs of opposite sides. This theorem can be written in trigonometric form which yields precisely the above identity for $p=2$. For $p=3$ formula~(\ref{eq:Ptol}) was f\/irst presented without proof by Glaisher at a 1880 conference and then proved in full generality by Hermite in his 1885 paper~\cite{Hermite}. The same paper contains a number of other amazing identities completely forgotten until the recent article~\cite{Johnson} by Johnson, where many interesting mathematical and historical details can be found. Furthermore, if we substitute $\sin(z)$ by~$z$ in every occurrence of $\sin$ in~(\ref{eq:Ptol}), we obtain an identity discovered by Gosper, Ismail and Zhang in~\cite{GIZ} and known as non-local derangement identity. It has been recently used by Feng, Kuznetsov and Yang to f\/ind new formulas for sums of products of generalized hypergeometric functions~\cite{Kuznetsov}.

When this paper was nearly f\/inished E.~Scheidegger published a preprint \cite{Scheidegger} that also deals with the hypergeometric dif\/ferential equation (\ref{eq:hyper-equation}) and builds largely on the works of N{\o}rlund~\cite{Norlund} and B\"{u}hring~\cite{Buehring92}. Scheidegger suggests a new basis of solutions around~$1$ and presents its series expansion. Main emphasis in~\cite{Scheidegger} is on the logarithmic cases (parameter dif\/ferences are integers). Scheidegger's work is motivated by applications in certain one-parameter families of Calabi--Yau manifolds, known as the mirror quartic and the mirror quintic.

\section{Results of N{\o}rlund and B\"{u}hring revisited}\label{section2}

\subsection{Fundamental solutions around unity}
The simple observation that nevertheless seems to be largely overlooked in the literature (except for \cite{Marichev,MK}) lies in the fact that N{\o}rlund's solutions to (\ref{eq:hyper-equation}) can be expressed in terms of Meijer's~$G$-function.
Let us remind its def\/inition f\/irst. Suppose $0\leq{m}\leq{q}$, $0\leq{n}\leq{p}$ are integers and~$\a$,~$\b$ are arbitrary complex vectors, such that $a_i-b_j\notin\N$ for all $i=1,\ldots,n$ and $j=1,\ldots,m$.
Meijer's $G$-function is def\/ined by the Mellin--Barnes integral of the form (see \cite[Section~5.3]{HTF1},
\cite[Chapter~1]{KilSaig}, \cite[Section~8.2]{PBM3} or \cite[Section~16.17]{NIST}),
\begin{gather}
G^{m,n}_{p,q} \left( z\,\vline\, \begin{matrix}\a\\\b\end{matrix} \right) :=
\frac{1}{2\pi{i}}
\int_{\mathcal{L}} \frac{\Gamma(b_1 + s)\cdots\Gamma(b_m + s)\Gamma(1-a_1 - s)\cdots\Gamma(1-a_n - s)z^{-s}}
{\Gamma(a_{n+1} + s)\cdots\Gamma(a_p + s)\Gamma(1-b_{m+1} - s)\cdots\Gamma(1-b_{q} - s)}ds,\label{eq:G-defined}
\end{gather}
where the contour $\L$ is a simple loop that separates the poles of the integrand of the form $b_{jl}=-b_j-l$, $l\in\N_0$, leaving them on the left from the poles of the form $a_{ik}=1-a_i+k$, $k\in\N_0$, leaving them on the right \cite[Section~1.1]{KilSaig}. It may have one of the three forms $\L_{-}$, $\L_{+}$ or $\L_{i\gamma}$ described below. Choose any
\begin{gather*}
\begin{split}
& \varphi_1<\min\{-\Im{b_1},\ldots,-\Im{b_m},\Im(1-a_1),\ldots,\Im(1-a_n)\},
\\
& \varphi_2>\max\{-\Im{b_1},\ldots,-\Im{b_m},\Im(1-a_1),\ldots,\Im(1-a_n)\}
\end{split}
\end{gather*}
and arbitrary real $\gamma$. The contour $\L_{-}$ is a left
loop lying in the horizontal strip $\varphi_1\leq\Im{s}\leq\varphi_2$. It starts at the point
$-\infty+i\varphi_1$, terminates at the point $-\infty+i\varphi_2$ and coincides with the boundary of the strip for suf\/f\/iciently large~$|s|$. Similarly, the contour~$\L_{+}$ is a right loop lying in the same strip, starting at the point $+\infty+i\varphi_1$ and terminating at the point $+\infty+i\varphi_2$. Finally, the
contour $\L_{i\gamma}$ starts at $\gamma-i\infty$, terminates at $\gamma+i\infty$ and coincides with the line $\Re{s}=\gamma$ for
all suf\/f\/iciently large $|s|$. The power function~$z^{-s}$ is def\/ined on the Riemann surface of the logarithm, so that
\begin{gather*}
z^{-s}=\exp(-s\{\log|z|+i\arg(z)\})
\end{gather*}
and $\arg(z)$ is allowed to take any real value. Hence,
$G^{m,n}_{p,q}(z)$ is also def\/ined on the Riemann surface of the
logarithm. In order that the above def\/inition be consistent one needs to prove that the value of the
integral remains intact if convergence takes place for several dif\/ferent contours. Alternatively, one may split the parameter
space into nonintersecting subsets and stipulate which contour should be used in each subset. Another key issue that must be
addressed with regard to the above def\/inition is whether the integral in (\ref{eq:G-defined}) equals the sum of residues of the integrand and, if yes, on which side of~$\L$ the residues are to be counted. This is important for both theoretical considerations (expressing $G$-function in terms of hypergeometric functions) and especially for actually computing the value of $G$-function (although numerical contour integration can also be employed). In this paper we will only need the $G$-function of the form $G^{m,n}_{p,p}$. For this particular type of $G$-function the solutions to the above problems seem to be rather complete. We placed further details regarding the def\/inition and answers to the above questions for $G^{m,n}_{p,p}$ in Appendix~\ref{appendixA}.

A simple property of Meijer's $G$-function implied by its def\/inition~(\ref{eq:G-defined}) which will be frequently used without further mentioning is given by \cite[(8.2.2.15)]{PBM3}
\begin{gather}\label{eq:Gtimespower}
z^{\alpha}G^{m,n}_{p,q} \left( z\,\vline\,\begin{matrix}\a\\\b\end{matrix} \right)=G^{m,n}_{p,q} \left( z\,\vline\,\begin{matrix}\a+\alpha\\ \b+\alpha\end{matrix} \right),
\end{gather}
where $\alpha\in\C$ and $\a+\alpha$ is understood as $(a_1+\alpha,\ldots,a_p+\alpha)$.

The functions $G^{p,0}_{p,p}$ and $G^{2,p}_{p,p}$ will play a particularly important role in this paper, so that we found it useful to cite the known explicit expressions for small~$p$. If $p=1$ then \cite[formula~(8.4.2.3)]{PBM3}
\begin{gather*}%\label{eq:G1011}
G^{1,0}_{1,1} \left( z\,\vline\,\begin{matrix}b\\ a\end{matrix} \right)=
\frac{z^a(1-z)_{+}^{b-a-1}}{\Gamma(b-a)},
\end{gather*}
where $(x)_+=x$ for $x\ge0$ and $0$ otherwise. If $p=2$ we have \cite[formula~(8.4.49.22)]{PBM3}
\begin{gather*}%\label{eq:G2022}
G^{2,0}_{2,2} \left( z\, \vline\, \begin{matrix}b_1,b_2
\\ a_1,a_2\end{matrix} \right)=\frac{z^{a_2}(1-z)_{+}^{b_1+b_2-a_1-a_2-1}}{\Gamma(b_1+b_2-a_1-a_2)}
\, {}_2F_{1} \left(\begin{matrix}b_1-a_1,b_2-a_1
\\ b_1+b_2-a_1-a_2\end{matrix} ;1-z\right),
\end{gather*}
and \cite[formula~(8.4.49.20)]{PBM3}
\begin{gather*}%\label{eq:G2222}
G^{2,2}_{2,2} \left( z\,\vline\, \begin{matrix}b_1,b_2
\\ a_1,a_2\end{matrix} \right)=\frac{z^{a_1}\Gamma(1+a_1-b_1)\Gamma(1+a_1-b_2)\Gamma(1+a_2-b_2)\Gamma(1+a_2-b_1)}{\Gamma(2+a_1+a_2-b_1-b_2)}
\\
\hphantom{G^{2,2}_{2,2} \left( z\,\vline\, \begin{matrix}b_1,b_2
\\ a_1,a_2\end{matrix} \right)=}{}
\times{_2F_{1}} \left(\begin{matrix}1+a_1-b_1,1+a_1-b_2\\2+a_1+a_2-b_1-b_2\end{matrix};1-z\right).
\end{gather*}
If $p=3$ we have \cite[formula~(8.4.51.2)]{PBM3}{\samepage
\begin{gather*}%\label{eq:G3033}
\begin{split}
& G^{3,0}_{3,3} \left( z\,\vline\,\begin{matrix}b_1,b_2,b_3\\ a_1,a_2,a_3\end{matrix} \right)
=\frac{z^{a_1+a_2-b_1-1}(1-z)_{+}^{b_1+b_2+b_3-a_1-a_2-a_3-1}}{\Gamma(b_1+b_2+b_3-a_1-a_2-a_3)}\\
& \qquad{} \times{F_3}(b_1-a_2,b_3-a_3;b_1-a_1,b_2-a_3;b_1+b_2+b_3-a_1-a_2-a_3;1-1/z,1-z),
\end{split}
\end{gather*}
where $F_3$ is Appell's hypergeometric function of two variables \cite[index of functions]{PBM3}.}

Having made these preparations we can formulate N{\o}rlund's result regarding a fundamental solution of~(\ref{eq:hyper-equation}) in the neighborhood of~$1$. First, introduce the following notation:
\begin{gather*}
\psi_m=\sum_{i=1}^m (b_i-a_i),\qquad 1\leq m\leq p,\nonumber\\
 \a=(a_1,\ldots,a_p),\qquad \b=(b_1,\ldots,b_p), \qquad \a_{[k]}=(a_1,\ldots,a_{k-1},a_{k+1},\ldots,a_p),\label{eq:notation}\\
\sin(\a)=\sin{a_1}\sin{a_2}\cdots\sin{a_p},\qquad \Gamma(\a)=\Gamma(a_1)\Gamma(a_2)\cdots\Gamma (a_p)
\end{gather*}
and let $\a_{[k,s]}$ denote the vector $\a$ with the elements $a_k$ and $a_s$
removed. We write $\Re(\a)>0$ for $\Re(a_i)>0$, $i=1,\dots,p$. The next theorem is implicit in~\cite{Norlund}.

\begin{Theorem}[N{\o}rlund]\label{th:N2}
Suppose $k,s\in\{1,\ldots,p\}$ and
\begin{gather*}
u_s(z)=G^{p,0}_{p,p} \left( z\, \vline\, \begin{matrix}\b\\ \a\end{matrix} \right),\qquad u_k(z)=G^{2,p}_{p,p} \left( z\,\vline\,\begin{matrix}\b \\ a_k,a_s,\a_{[k,s]}\end{matrix} \right)\qquad \text{for} \quad k\neq{s}.
\end{gather*}
Then the set $\{u_k(z)\}_{k=1}^p$ forms the fundamental system of solutions of \eqref{eq:hyper-equation} in the neighborhood of $z=1$ and
\begin{gather}
\sin(\pi\psi_p)\frac{z^{a_s}\Gamma(1-\b+a_s)}{\Gamma(1-\a_{[s]}+a_s)}
\, {}_{p}F_{p-1}\left(\begin{matrix}1-\b+a_s\\ 1-\a_{[s]}+a_s\end{matrix};z\right)\nonumber\\
\qquad{} =\pi G^{p,0}_{p,p} \left( z\,\vline\, \begin{matrix}\b
\\ \a\end{matrix} \right)
-\frac{1}{\pi}\sum_{\substack{k=1\\k\ne{s}}}^{p} \frac{\sin(\pi(\b-a_k))}{\sin(\pi(\a_{[k,s]}-a_k))}
G^{2,p}_{p,p} \left( z\,\vline\,\begin{matrix}\b\\ a_k,a_s,\a_{[k,s]}\end{matrix} \right).\label{eq:Norlund5.40}
\end{gather}
\end{Theorem}

 \begin{proof} Formula (\ref{eq:Norlund5.40}) is a rewriting of \cite[(5.40)]{Norlund}. Indeed, f\/irst compare \cite[(2.44)]{Norlund} with the def\/inition of $G$-function (\ref{eq:G-defined}) to see how N{\o}rlund's~$\xi_n$ is expressed by~$G^{p,0}_{p,p}$. Alternatively, this connection follows on comparing the Mellin transforms (\ref{eq:GMellin}) and \cite[(2.18)]{Norlund}. Further, use \cite[(5.45)]{Norlund} to expand N{\o}rlund's~$\varphi_s$ and~\cite[(5.7)]{Norlund} to express~$y_{k,s}$ in terms of~$G^{2,p}_{p,p}$ (see also~(\ref{eq:Norl-y12}) below).
 \end{proof}

\begin{Remark}
Formula~(\ref{eq:Norlund5.40}) can now be viewed as the ref\/lection of the fact that any $p+1$ solutions must
be linearly dependent so that any solution in the neighborhood of $0$ can be expressed in terms of the fundamental set of solutions around~$1$. This formula extends the connection formula (\ref{eq:2F1connection}) for the Gauss hypergeometric function to which it reduces when $p=2$. A closely related formula has been also discovered by B\"{u}hring in two papers~\cite{Buehring87} (for $p=3$) and~\cite{Buehring92} (for general~$p$). See formula~(\ref{eq:Buehring1}) below.
\end{Remark}

\begin{Remark}
Marichev and Kalla reproduced formula \cite[(5.40)]{Norlund} in \cite[(11)]{MK} and gave expressions for its components in terms of~$G$-function in \cite[(27), (33), (36)]{MK}.
 \end{Remark}

Expansion of the solution $G^{p,0}_{p,p}$ in terms of the fundamental solutions around $z=0$ that compliments~(\ref{eq:Norlund5.40}) coincides with the well-known expansion obtained from the def\/inition of $G$-function by the residue theorem. Namely, if the elements of the vector $\a$ are dif\/ferent modulo integers, formula (\ref{eq:sumresleft}) %from the Appendix
applied to the function~$G^{p,0}_{p,p}$ takes the form (see also \cite[formula~(34)]{Marichev} and~\cite[8.2.2.3]{PBM3}):
\begin{gather}\label{eq:Gp0pp}
G^{p,0}_{p,p} \left( z\,\vline\, \begin{matrix}\b\\ \a\end{matrix} \right)=
\sum\limits_{k=1}^{p}z^{a_k}\frac{\Gamma(\a_{[k]}-a_k)}{\Gamma(\b-a_k)}
\,{}_pF_{p-1} \left( \begin{matrix}1-\b+a_k \\ 1-\a_{[k]}+a_k\end{matrix}\,\vline\, z\right).
\end{gather}
According to Theorem~\ref{th:sumres}(a) this formula holds for $|z|<1$.

\subsection[Expansion of $G^{p,0}_{p,p}$ in the neighborhood of unity]{Expansion of $\boldsymbol{G^{p,0}_{p,p}}$ in the neighborhood of unity}

Among many other results contained in \cite{Norlund} N{\o}rlund showed that the series
\begin{gather}\label{eq:Norl-xi}
G^{p,0}_{p,p} \left( z\,\vline\,\begin{matrix}\b
\\ \a\end{matrix} \right)=\frac{z^{a_k}(1-z)^{\psi_p-1}}{\Gamma(\psi_p)}\sum_{n=0}^{\infty}\frac{g_p^k(n)}{(\psi_p)_n}(1-z)^n
\end{gather}
represents a solution in the neighborhood of $z=1$ corresponding to the local exponent~$\psi_p-1$ (see~(\ref{eq:notation}) for the def\/inition of $\psi_p$) if this number is not a negative integer. Formula~(\ref{eq:Norl-xi}) holds in the disk $|1-z|<1$ for all $-\psi_p\notin\N_0$ and each $k=1,2,\ldots,p$. For $-\psi_p=l\in\N_0$ we have by taking limit in~(\ref{eq:Norl-xi})
(see \cite[(1.34)]{Norlund}):
\begin{gather}\label{eq:Norlund2}
G^{p,0}_{p,p} \left( z\,\vline\, \begin{matrix}\b\\\ a\end{matrix} \right)
=z^{a_k}\sum\limits_{n=0}^{\infty}\frac{g^{k}_{p}(n+l+1)}{n!}(1-z)^n,\qquad k=1,2,\ldots,p.
\end{gather}
The value of $k$ in (\ref{eq:Norl-xi}) and (\ref{eq:Norlund2}) can be chosen arbitrarily from the set~$\{1,2,\ldots,p\}$. This choice af\/fects the f\/irst factor~$z^{a_k}$ and the coef\/f\/icients~$g_p^k(n)$ while the left-hand side is of course independent of~$k$.

Let us present N{\o}rlund's formulas for the coef\/f\/icients $g^{k}_{p}(n)$. First, by applying
the Frobenius ansatz to the dif\/ferential equation~(\ref{eq:hyper-equation}) N{\o}rlund demonstrated that~$g^k_p(n)$
satisfy the $p$-th order dif\/ference equation with polynomial coef\/f\/icients given by
\begin{gather}\label{eq:Norl-recur}
\sum_{i=0}^{p-1}P_{p-i}(n\,|\,\a,\b)g^{k}_{p}(n+i)=0,\qquad n=1,2,\ldots,
\end{gather}
where $P_{m}(z\,|\,\a,\b)$ is a polynomial in $z$ of degree $m$ with coef\/f\/icients dependent on the pa\-ra\-me\-ters~$\a$ and~$\b$. It is expressed in terms of the polynomials
\begin{gather*}
Q(z)=\prod_{i=1}^{p}(z-a_i),\qquad R(z)=\prod_{i=1}^{p}(z+1-b_i)
\end{gather*}
as follows \cite[(1.28)]{Norlund}:
\begin{gather*}
 P_1(z\,|\,\a,\b)=p-1+z,\\
P_{j}(z\,|\,\a,\b)=\frac{(-1)^{j}}{(p-j-1)!}\Delta^{p-j-1}Q(\psi_p+a_k+z)
-\frac{(-1)^{j}}{(p-j)!}\Delta^{p-j}R(\psi_p-1+a_k+z)
\end{gather*}
for $j=2,\ldots,p-1$, and
\begin{gather*}
P_p(z\,|\,\a,\b)=-(-1)^{p}R(\psi_p-1+a_k+z).
\end{gather*}
Here the dif\/ference is understood as the forward dif\/ference $\Delta{Q(z)}=Q(z+1)-Q(z)$,
$\Delta^{m}{Q(z)}=\Delta(\Delta^{m-1}{Q(z)})$. The initial values
$g^k_p(0)$, $g^k_p(1)$, $\ldots$, $g^k_p(p-1)$ for the recurrence~(\ref{eq:Norl-recur}) are found by solving the next triangular system:
\begin{gather*}
\begin{split}
&g^{k}_{p}(0)=1,\\
&g^{k}_{p}(1)+P_{2}(2-p\,|\,\a,\b)g^{k}_{p}(0)=0,\\
&2g^{k}_{p}(2)+P_{2}(3-p\,|\,\a,\b)g^{k}_{p}(1)+P_{3}(3-p\,|\,\a,\b)g^{k}_{p}(0)=0,\\
&3g^{k}_{p}(3)+P_{2}(4-p\,|\,\a,\b)g^{k}_{p}(2)+P_{3}(4-p\,|\,\a,\b)g^{k}_{p}(1)+P_{4}(4-p\,|\,\a,\b)g^{k}_{p}(0)=0,\\
&\cdots\cdots\cdots\cdots\cdots\cdots\cdots\cdots\cdots\cdots\cdots\cdots\cdots\cdots\cdots\cdots\cdots\cdots\cdots\cdots\cdots\cdots\cdots\cdots\cdots\cdots\\
&(p-1)g^{k}_{p}(p-1)+P_{2}(0\,|\,\a,\b)g^{k}_{p}(p-2)+P_{3}(0\,|\,\a,\b)g^{k}_{p}(p-3)+\!\cdots\!+P_{p}(0\,|\,\a,\b)g^{k}_{p}(0)=0.
\end{split}
\end{gather*}
Another method to compute the coef\/f\/icients $g^k_{p}(n)$ discovered
by N{\o}rlund is the following recurrence in~$p$~\cite[(2.7)]{Norlund}:
\begin{gather}
g^{1}_{1}(n)= \begin{cases}1,& n=0,\\ 0,& n\ge1,\end{cases} \qquad g^{p}_{p}(n)=\sum_{j=0}^{n}\frac{(b_p-a_k)_{n-j}}{(n-j)!}(\psi_{p-1}+j)_{n-j}g^k_{p-1}(n).\label{Norl-rec-p}
\end{gather}
The last formula can be applied for any $k=1,2,\ldots,p-1$ without af\/fecting the left-hand side. The value of $g^{k}_{p}(n)$ for $k\ne{p}$ can then be obtained by exchanging the roles of $a_p$ and $a_k$ in the resulting expression for $g^{p}_{p}(n)$.
Alternatively, N{\o}rlund gives the following connection formula \cite[(1.35)]{Norlund}:
\begin{gather*}
g^{k}_{p}(n)=\sum\limits_{j=0}^{n}\frac{(a_k-a_l)_{n-j}}{(n-j)!}(\psi_p+j)_{n-j}g^{l}_{p}(j),\qquad k,l=1,\ldots,p,\qquad k\ne{l}.
\end{gather*}
Furthermore, he solved the above recurrence to obtain
\cite[(2.11)]{Norlund}:
\begin{gather}\label{eq:Norlund-explicit}
g^{p}_{p}(n)=\sum\limits_{0\leq{j_{1}}\leq{j_{2}}\leq\cdots\leq{j_{p-2}}\leq{n}}
\prod\limits_{m=1}^{p-1}\frac{(\psi_m+j_{m-1})_{j_{m}-j_{m-1}}}{(j_{m}-j_{m-1})!}(b_{m+1}-a_{m})_{j_{m}-j_{m-1}},
\end{gather}
where, as before, $\psi_m=\sum_{i=1}^{m}(b_i-a_i)$ and $j_0=0$, $j_{p-1}=n$.
This formula shows explicitly that $g^{p}_{p}(n)$ does not depend on $a_p$. Observe that
\begin{itemize}\itemsep=0pt
\item the coef\/f\/icient $g^{k}_{p}(n)$ is a symmetric polynomial in the components of~$\b$ and $\a_{[k]}$ (separately);
\item the summation in (\ref{eq:Norlund-explicit}) is over all Young diagrams that f\/it $n\times{(p-2)}$ box.
\end{itemize}
Symmetry follows from expansion (\ref{eq:Norl-xi}) and the invariance of $G$-function with respect to permutation the elements of $\a$ and $\b$. In Section~\ref{section3} below we f\/ind another way to compute the coef\/f\/icients $g_{p}^{k}(n)$ in terms of generalized Bernoulli polynomials and we further give explicit formulas for~$g_{p}^{k}(n)$, $n=1,2,3$. Connection of these coef\/f\/icients to combinatorics deserves further investigation.

\begin{Remark}
Formulas (\ref{eq:Norlund5.40}) and (\ref{eq:Norl-xi}) have been reproduced by Marichev \cite[(4), (17)]{Marichev} and Marichev and Kalla in \cite[(12), (27)]{MK}. Surprisingly, these important formulas seem to have been largely overlooked in the special function literature and have never been included in any textbook on hypergeometric functions, except for the reference book \cite{PBM3} by Prudnikov, Brychkov and Marichev. However, even in this book the description of the behavior of $G^{p,0}_{p,p}(z)$ in the neighborhood of $z=1$ contains an incorrect assertion in the case of non-positive integer~$\psi_p$, see \cite[Section~8.2.2.59]{PBM3}. None of the identities presented in Section~\ref{section3} below are contained in \cite{Marichev,MK,PBM3}.
\end{Remark}

\begin{Remark}
 Another very prolif\/ic line of research that involves $G$-function forms an important part of the statistics literature and began with 1932 paper of Wilks \cite{Wilks}, where he observed that the moments of many likelihood ratio criteria in multivariate hypothesis testing are expressed in terms of product ratios of gamma functions. Wilks introduced two types of integral equations, of which ``type B'' is essentially equation~(\ref{eq:GMellin}) of the Appendix. He also noticed that the solution of ``type B integral equation'' represents the probability density of the product of independent beta distributed random variables. Wilks' ideas were elaborated in dozens of papers that followed, mainly concerned with calculating and approximating the solution of ``type B integral equation''. We will just mention a few key contributions, where an interested reader may f\/ind further references. In his 1939 paper~\cite{Nair} Nair derived the dif\/ferential equation (\ref{eq:hyper-equation}) satisf\/ied by the solution of ``type B integral equation'' thus demonstrating that $G^{p,0}_{p,p}$ satisf\/ies~(\ref{eq:hyper-equation}). This happened long before N{\o}rlund wrote his paper~\cite{Norlund} and was also independent of Meijer's work, although Meijer already introduced $G$-function as a linear combination of hypergeometric functions in 1936. Nair also was the f\/irst (among researchers in statistics) to apply the inverse Mellin transform to the right-hand side of~(\ref{eq:GMellin})~-- the approach further developed by Consul in a series of papers between 1964 and 1969, where the connection to Meijer's $G$-function was f\/irst observed. See~\cite{Consul} and references therein. Mellin transform technique was then utilized in a number of papers by Mathai who later rediscovered N{\o}rlund's coef\/f\/icients in a form similar to~(\ref{eq:Norlund-explicit}) in~\cite{Mathai84}. Mathai's and other contributions until 1973 are described in his survey paper~\cite{Mathai73}. In the same period Springer and Thompson independently expressed the densities of products and ratios of gamma and beta distributed random variables in terms of Meijer's $G$-function, see~\cite{Springer-Thompson}. Davis~\cite{Davis} presented the matrix form of the dif\/ferential equation for $G^{p,0}_{p,p}$ and suggested the series solution similar to~(\ref{eq:Norl-xi}) with coef\/f\/icients found by certain recursive procedure. Another noticeable contribution is due to Gupta and Tang who found two series expansions for $G^{p,0}_{p,p}$ (again using ``type B integral equation'' terminology), one of them equivalent to (\ref{eq:Norl-xi}), and rediscovered recurrence relation (\ref{Norl-rec-p}). See \cite{TangGupta84,TangGupta86} and references there. This line of research continues until today, as evidenced, for example, by a series of papers by Carlos Coelho with several co-authors, see \cite{CoelhoArnold} and references there. Furthermore, Charles Dunkl rediscovered N{\o}rlund's recurrence (\ref{eq:Norl-recur}) in his preprint \cite{Dunkl} dated 2013, where he again considers the probability density of a product of beta distributed random variables. Independently, probability distribution with $G$-function density has been found to be the stationary distribution of certain Markov chains considered, for example, in actuarial science and is known as Dufresne law in this context. See \cite{ChamLetac,Dufresne2010} and references therein. Let us also mention that $G$-function popped up recently in the random matrix theory as correlation kernel of a determinantal point process that governs singular values of products of~$M$ rectangular random matrices with independent complex Gaussian entries~\cite{AIK}. Curiously enough, none of these authors cited N{\o}rlund's work. Let us also mention that in our recent paper~\cite{KPCMFT} conditions are given under which $G_{p,p}^{p,0}(e^{-x})$ is inf\/initely divisible distribution on~$[0,\infty)$.
\end{Remark}

\subsection[Expansion of $G^{2,p}_{p,p}$ in the neighborhood of unity]{Expansion of $\boldsymbol{G^{2,p}_{p,p}}$ in the neighborhood of unity}

Formula \cite[(5.7)]{Norlund} shows that N{\o}rlund's function $y_{1,2}(x)$ def\/ined by \cite[formula~(5.2)]{Norlund} is expressed in terms of Meijer's $G$-function as follows
\begin{gather}\label{eq:Norl-y12}
y_{1,2}(x)=G^{2,p}_{p,p} \left( x\,\vline\,\begin{matrix}1-\alpha_1,\ldots,1-\alpha_p
\\ \gamma_1,\gamma_2,\gamma_3,\ldots,\gamma_p\end{matrix} \right).
\end{gather}
To def\/ine $y_{k,s}(x)$ the roles of $\gamma_1$, $\gamma_2$ are exchanged with those of $\gamma_k$, $\gamma_s$. N{\o}rlund found
several expansions of the functions $y_{k,s}(x)$ in hypergeometric polynomials which we cite below. Changing N{\o}rlund's notation to ours according to the rule $x\mapsto{z}$, $\gamma_i\mapsto{a_i}$, $\alpha_i\mapsto{1-b_i}$, formula \cite[(5.3)]{Norlund} takes the form
\begin{gather}\sin\pi(a_s-a_i)
G^{2,p}_{p,p} \left( z\,\vline\,\begin{matrix}\b\\a_s,a_i,\a_{[s,i]}\end{matrix} \right)+
\sin\pi(a_i-a_k)G^{2,p}_{p,p} \left( z\, \vline\, \begin{matrix}\b\\a_i,a_k,\a_{[i,k]}\end{matrix} \right)
\nonumber\\
\hphantom{\sin\pi(a_s-a_i) G^{2,p}_{p,p} \left( z\,\vline\,\begin{matrix}\b\\a_s,a_i,\a_{[s,i]}\end{matrix} \right)}{}
+\sin\pi(a_k-a_s)G^{2,p}_{p,p} \left( z\, \vline\,\begin{matrix}\b\\a_k,a_s,\a_{[k,s]}\end{matrix} \right)=0\label{eq:G2pppNorlund5.3}
\end{gather}
for any distinct values of $s,i,k\in\{1,\ldots,p\}$, where $\a_{[k,s]}$ denotes the vector $\a$
with elements $a_k$ and $a_s$ removed. Expansions \cite[(5.20), (5.22), (5.23), (5.31)]{Norlund} written in terms of
$G$-function take the form
\begin{gather}
G^{2,p}_{p,p} \left( z\,\vline\,\begin{matrix}\b
\\a_k,a_s,\a_{[k,s]}\end{matrix} \right)=\frac{z^{a_k}\Gamma(1-\b+a_k)\Gamma(1-b_1+a_s)\Gamma(1-b_2+a_s)}{\Gamma(1-\a_{[k,s]}+a_k)\Gamma(2+a_k+a_s-b_1-b_2)}
\notag\\
\quad{}\times\sum_{n=0}^{\infty}\frac{(1-b_1+a_k)_n(1-b_2+a_k)_n}{(2+a_k+a_s-b_1-b_2)_n}\,
{}_{p-1}F_{p-2}\left(\begin{matrix}-n,1-\b_{[1,2]}+a_k\\ 1-\a_{[k,s]}+a_k\end{matrix}\,\vline\, z\right)\label{eq:Norlund5.20}
\\
{}=\frac{z^{a_k}\Gamma(1-a_k+a_s)\Gamma(1-\b+a_k)}{\Gamma(1-\a_{[k,s]}+a_k)}
\sum_{n=0}^{\infty}\frac{(1-b_1+a_k)_n}{(1-b_1+a_s)_{n+1}}
\, {}_{p}F_{p-1}\left(\begin{matrix}-n,1-\b_{[1]}+a_k\\1,
1-\a_{[k,s]}+a_k\end{matrix}\,\vline\,z\right)\!\!\!\!\!\label{eq:Norlund5.22}
\\
{}=\frac{z^{a_k}\Gamma(1-b_1+a_s)\Gamma(1-\b_{[1]}+a_k)}{\Gamma(1-\a_{[k,s]}+a_k)}
\sum_{n=0}^{\infty}\frac{(1-a_s+a_k)_n}{n!(1-b_1+a_k+n)}\nonumber\\
\quad{}\times
{}_{p}F_{p-1}\left(\begin{matrix}-n,1-\b_{[1]}+a_k\\ 1-\a_{[k]}+a_k\end{matrix}\,\vline\,z\right)\label{eq:Norlund5.23}
\\
{} =\frac{z^{a_k}\pi(a_k-a_s)\Gamma(1-\b+a_k)}{\sin(\pi(a_k-a_s))\Gamma(1-\a_{[k]}+a_k)}
\sum_{n=0}^{\infty}\frac{(1-a_s+a_k)_n}{(n+1)!}\,
{}_{p+1}F_{p}\left(\begin{matrix}-n,1-\b+a_k\\ 1,
1-\a_{[k]}+a_k\end{matrix}\, \vline\, z\right).\label{eq:Norlund5.31}
\end{gather}
The series in (\ref{eq:Norlund5.20}) converges in the disk $|z-1|<1$ if $\Re(1-\b_{[1,2]}+a_s)>0$, the series in~(\ref{eq:Norlund5.22}),~(\ref{eq:Norlund5.23}) converge in the same disk if $\Re(1-\b_{[1]}+a_s)>0$ and, f\/inally,
(\ref{eq:Norlund5.31}) converges in $|z-1|<1$ if $\Re(1-\b+a_s)>0$. In all cases, it is also required that none of
the gamma functions in the numerator had poles. Further, N{\o}rlund found two expansions of his function~$y_{k,s}(x)$ in
powers of $1-z$, which we will need below. Written in our notation, formulas \cite[(5.35), (5.36)]{Norlund} read
\begin{gather}\label{eq:G2pppNorlund}
G^{2,p}_{p,p} \left( z\, \vline\, \begin{matrix}\b\\ a_k,a_s,\a_{[k,s]}\end{matrix} \right)
=z^{a_s}\sum_{n=0}^{\infty}D_n^{[k,s]}(1-z)^n,
\end{gather}
where the coef\/f\/icients $D_n^{[k,s]}$ are given by
\begin{gather}
\notag
D_n^{[k,s]}=\frac{\Gamma(1-\b+a_k)\Gamma(1-b_1+a_s+n)\Gamma(1-b_2+a_s+n)}{\Gamma(1-\a_{[k,s]}+a_k)\Gamma(2+a_k+a_s-b_1-b_2+n)n!}
\\
\hphantom{D_n^{[k,s]}=}{}
\times\sum_{j=0}^{\infty}\frac{(1-b_1+a_k)_j(1-b_2+a_k)_j}{j!(2+a_k+a_s-b_1-b_2+n)_j}\,
{}_{p-1}F_{p-2}\left( \begin{matrix}-j,1-\b_{[1,2]}+a_k\\1-\a_{[k,s]}+a_k\end{matrix} \right)\label{eq:Norlund5.35}
\\
\hphantom{D_n^{[k,s]}}{}
=\frac{\Gamma(1-\b+a_k)\Gamma(1-a_k+a_s+n)}{\Gamma(1-\a_{[k,s]}+a_k)n!}\sum_{j=0}^{\infty}\frac{(1-b_1+a_k)_j}{(1-b_1+a_s+n)_{j+1}}\nonumber\\
\hphantom{D_n^{[k,s]}=}{}\times
 {}_{p}F_{p-1} \left( \begin{matrix}-j,1-\b_{[1]}+a_k\\ 1,1-\a_{[k,s]}+a_k\end{matrix} \right).\label{eq:Norlund5.36}
\end{gather}
Here and below ${}_pF_{p-1}$ without an argument is understood as ${}_pF_{p-1}(1)$. The series (\ref{eq:G2pppNorlund}) converges in $|z-1|<1$ if $\Re(1-\b_{[1,2]}+a_s)>0$. The same condition suf\/f\/ices for convergence of~(\ref{eq:Norlund5.35}), while (\ref{eq:Norlund5.36}) converges if $\Re(1-\b_{[1]}+a_s)>0$. Uniqueness of power series coef\/f\/icients implies equality of the coef\/f\/icients in~(\ref{eq:Norlund5.35}) and~(\ref{eq:Norlund5.36}). For $p=2$ this equality is nothing but Gauss formula for~${}_2F_1(1)$. For $p=3$ we get after some
simplif\/ications and renaming variables
\begin{gather*}{}_3F_2 \left( \begin{matrix}\alpha_1,\alpha_2,\alpha_3\\\beta_1,\beta_2\end{matrix} \right)
 = \frac{\Gamma(\beta_2-\alpha_1-\alpha_2+1)\Gamma(\beta_2)}{\Gamma(\beta_2-\alpha_1)\Gamma(\beta_2-\alpha_2)}
\sum_{j=0}^{\infty}\frac{(\alpha_1)_j}{(\beta_2-\alpha_2)_{j+1}}\,
{}_3F_2 \left( \begin{matrix}-j,\alpha_2,\beta_1-\alpha_3\\\beta_1,1\end{matrix} \right).
\end{gather*}

\subsection[Connection to B\"{u}hring expansion of ${}_pF_{p-1}$]{Connection to B\"{u}hring expansion of $\boldsymbol{{}_pF_{p-1}}$}

In his 1992 paper B\"{u}hring found the representation \cite[Theorem~2]{Buehring92}
\begin{gather}\label{eq:Buehring1}
\frac{\Gamma(\al)}{\Gamma(\be)}{}_{p}F_{p-1}\left( \begin{matrix}\al\\ \be\end{matrix};z \right)
=(1-z)^{\nu}\sum_{n=0}^{\infty}f_p(n|\al,\be)(1-z)^n+\sum_{n=0}^{\infty}h_p(n|\al,\be)(1-z)^n,
\end{gather}
where $\nu=\sum\limits_{k=1}^{p-1}\beta_k-\sum\limits_{k=1}^{p}\alpha_k$. He also derived explicit formulas for the coef\/f\/icients $f_p(n|\al,\be)$ and $h_p(n|\al,\be)$ to be given below. On setting $\al=1-\b+a_s$, $\be=1-\a_{[s]}+a_s$, $\nu=\psi_p-1$ and denoting
\begin{gather*}%\label{eq:Buehring11}
f_p^s(n)\equiv f_p(n|1-\b+a_s,1-\a_{[s]}+a_s),
\\
%\label{eq:Buehring12}
h_p^s(n)\equiv h_p(n|1-\b+a_s,1-\a_{[s]}+a_s),
\end{gather*}
B\"{u}hring's formula takes the form
\begin{gather}
\frac{\Gamma(1-\b+a_s)}{\Gamma(1-\a_{[s]}+a_s)}\, {}_pF_{p-1}\left(\begin{matrix}1-\b+a_s\\ 1-\a_{[s]}+a_s\end{matrix};z\right)\nonumber\\
\qquad{}=(1-z)^{\psi_p-1}\sum_{n=0}^{\infty}f_p^s(n)(1-z)^n+\sum_{n=0}^{\infty}h_p^s(n)(1-z)^n.\label{eq:Buehring13}
\end{gather}
Expansion of the form (\ref{eq:Buehring1}) is unique as long as $\nu$ is not an integer. Hence, in view of~(\ref{eq:Gp0pp}) and~(\ref{eq:Norlund5.35}), identity~(\ref{eq:Buehring13}) is equivalent to formula~(\ref{eq:Norlund5.40}). From~(\ref{eq:Norlund5.40}) we conclude that
\begin{gather}\label{eq:singularconnection}
(1-z)^{\psi_p-1}\sum_{n=0}^{\infty}f_p^s(n)(1-z)^n=\frac{\pi z^{-a_s}}{\sin(\pi\psi_p)}
G^{p,0}_{p,p} \left( z\, \vline\, \begin{matrix}\b
\\\a\end{matrix} \right) \qquad \text{and}\\
\label{eq:regularconnection}
\sum_{n=0}^{\infty}h_p^s(n)(1-z)^n=\frac{-z^{-a_s}}{\pi\sin(\pi\psi_p)}\sum_{\substack{k=1\\k\ne{s}}}^{p} \frac{\sin(\pi(\b-a_k))}{\sin(\pi(\a_{[k,s]}-a_k))}
G^{2,p}_{p,p} \left( z\, \vline\,\begin{matrix}\b-a_s\\a_k,a_s,\a_{[k,s]}\end{matrix} \right)
\end{gather}
for each $s\in\{1,\ldots,p\}$. Comparing (\ref{eq:singularconnection}) with (\ref{eq:Norl-xi}) we arrive at
\begin{gather}\label{eq:Buehr-Norl-eq}
f^s_p(n)=\frac{\Gamma(1-\psi_p)}{(\psi_p)_n}g^s_p(n),
\end{gather}
where Euler's ref\/lection formula has been used. Thus, B\"{u}hring
coef\/f\/icients $f_p(n|\al,\be)$ (denoted by $g_n(s)$ in~\cite{Buehring92}) are essentially the same as N{\o}rlund coef\/f\/icients $g^s_p(n)$ (denoted by $c^{(s)}_{n,p}$ in~\cite{Norlund}). Explicit representation for the coef\/f\/icients $f_p(n|\al,\be)$ found in~\cite[(2.9), (2.16)]{Buehring92} after some renaming of variables and setting $s=p$ can be put into the form
\begin{gather*}
f_p^p(n) = \frac{\Gamma(1-\psi_p)}{(\psi_p)_n}
 \sum^{n}_{j_{p-2}=0}\frac{(-1)^{j_{p-2}}(\psi_p+a_p-b_{p-1}+j_{p-2})_{n-j_{p-2}}}{(n-j_{p-2})!}\\
\hphantom{f_p^p(n) =}{}
\times\sum^{j_{p-2}}_{j_{p-3}=0} \cdots \sum^{j_3}_{j_{2}=0}\sum^{j_2}_{j_{1}=0}
\prod\limits_{m=1}^{p-1}\frac{(\psi_m+j_{m-1})_{j_{m}-j_{m-1}}}{(j_{m}-j_{m-1})!}
\prod\limits_{m=1}^{p-2}(b_m-a_{m+1})_{j_{m}-j_{m-1}}\\
\hphantom{f_p^p(n)}{}
= \frac{\Gamma(1-\psi_p)}{(\psi_p)_n}
 \sum^{\infty}_{j_{1},j_{2},\ldots,j_{p-2}=0} \frac{(\psi_p+a_p-b_{p-1})_{n}}{(\psi_p + a_p - b_{p-1})_{j_{p-2}}}
\prod\limits_{m=1}^{p-1} \frac{(\psi_m)_{j_m}(-j_{m})_{j_{m-1}}}{(\psi_m)_{j_{m-1}}j_{m}!}\\
\hphantom{f_p^p(n) =}{}\times
 \prod\limits_{m=1}^{p-2} \frac{(b_m-a_{m+1})_{j_{m}}}{(1 - b_m + a_{m+1} - j_{m})_{j_{m-1}}},
\end{gather*}
where $\psi_m=\sum\limits_{i=1}^{m}(b_i-a_i)$ and $j_0=0$, $j_{p-1}=n$.
In view of~(\ref{eq:Norlund-explicit}), formula~(\ref{eq:Buehr-Norl-eq}) then leads to the following (presumably new) transformation for multiple hypergeometric series
\begin{gather*}
\sum^{n}_{j_{p-2}=0} \frac{(\psi_p+a_p-b_{p-1}+j_{p-2})_{n-j_{p-2}}}{(-1)^{j_{p-2}}(n-j_{p-2})!}\\
\qquad\quad{}\times
\sum^{j_{p-2}}_{j_{p-3}=0} \cdots \sum^{j_3}_{j_{2}=0}\sum^{j_2}_{j_{1}=0}
\prod\limits_{m=1}^{p-1} \frac{(\psi_m+j_{m-1})_{j_{m}-j_{m-1}}}{(j_{m}-j_{m-1})!}
\prod\limits_{m=1}^{p-2}(b_m-a_{m+1})_{j_{m}-j_{m-1}}\\
\qquad{} =\sum^{n}_{j_{p-2}=0} \cdots \sum^{j_3}_{j_{2}=0}\sum^{j_2}_{j_{1}=0}
\prod\limits_{m=1}^{p-1}\frac{(\psi_m+j_{m-1})_{j_{m}-j_{m-1}}}{(j_{m}-j_{m-1})!}(b_{m+1}-a_{m})_{j_{m}-j_{m-1}},
\end{gather*}
where $j_0=0$ and $j_{p-1}=n$. For $p=3$ this formula reduces to the identity
\begin{gather*}
{}_3F_2 \left(\begin{matrix}-n,\alpha_1,\alpha_2\\\beta_1,\beta_2\end{matrix}\right)
=\frac{(\beta_2-\alpha_2)_n}{(\beta_2)_n}\, {}_3F_2 \left(\begin{matrix}-n,\beta_1-\alpha_1,\alpha_2\\ \beta_1,1-\beta_2+\alpha_2-n\end{matrix}\right),
\end{gather*}
which is a guise of Sheppard's transformation~\cite[Corollary~3.3.4]{AAR}, also rediscovered by B\"{u}hring~\cite[(4.1)]{Buehring95}. For $p=4$ the corresponding result is given by~\cite[(4.2)]{Buehring95}
\begin{gather*}
\sum\limits_{k=0}^{n}\frac{(-k)_n(\alpha_1)_k(\alpha_2)_k}{(\beta_1)_k(\beta_2)_kk!}\,
{}_3F_2 \left(\begin{matrix}-k,\gamma_1,\gamma_2\\\alpha_1,\alpha_2\end{matrix}\right)\\
\qquad{}=\frac{(\beta_2-\alpha_2)_n}{(\beta_2)_n}\sum\limits_{k=0}^{n}\frac{(-k)_n(\alpha_2)_k(\beta_1-\alpha_1)_k}{(\beta_1)_k(1+\alpha_2-\beta_2-n)_kk!}\,
{}_3F_2 \left(\begin{matrix}-k,\gamma_1,\gamma_2\\\alpha_2,1+\alpha_1-\beta_1-k\end{matrix}\right).
\end{gather*}

In his paper \cite{Buehring92} B\"{u}hring mentioned that the structure of (\ref{eq:Buehring1}) ``was already given in the classic paper by N{\o}rlund \cite{Norlund}, but the coef\/f\/icients were not all known''. In his subsequent joint work~\cite{BS} with Srivastava is it said that the coef\/f\/icients $f_p^s(n)$ can be computed using N{\o}rlund's recurrence, ``but it is desirable to get an explicit representation'' which is indeed derived in~\cite{BS} in terms of some limit relations. Explicit formulas for the coef\/f\/icients $h_p^s(n)$
have been f\/irst found in~\cite{Buehring92}. Adopted to our notation B\"{u}hring's formulas \cite[(2.9), (2.16)]{Buehring92}, \cite[(2.13), (2.15)]{BS} read
\begin{gather}
 h_p^p(n)=
-\frac{\Gamma(\psi_p)\Gamma(1-b_p+a_p+n)\Gamma(1-b_{p-1}+a_p+n)}{(1-\psi_p)_{n+1}\Gamma(\psi_{p-1})
\Gamma(\psi_p-b_{p-1}+a_p)n!}\nonumber\\
\hphantom{h_p^p(n)=}{} \times
\sum\limits_{j_{p-2}=0}^{\infty} \frac{(\psi_p-n-1)_{j_{p-2}}}{(\psi_{p-1})_{j_{p-2}}(\psi_p-b_{p-1}+a_p)_{j_{p-2}}}\nonumber\\
\hphantom{h_p^p(n)=}{} \times
 \sum\limits_{j_{p-3}=0}^{j_{p-2}}\cdots\sum\limits_{j_{1}=0}^{j_{2}}
\prod\limits_{m=1}^{p-2}\frac{(\psi_m+j_{m-1})_{j_{m}-j_{m-1}}}{(j_{m}-j_{m-1})!}(b_m-a_{m+1})_{j_{m}-j_{m-1}},\label{eq:hp-explicit}
\end{gather}
where, as before, $\psi_m=\sum\limits_{i=1}^{m}(b_i-a_i)$ and $j_0=0$. Conditions for convergence for the outer series above have also
been found in~\cite{Buehring92} and~\cite{BS} and are given by
\begin{gather*}
\Re(1-b_i+a_p+n)>0 \qquad \text{for} \quad i=1,2,\ldots,p-2.
\end{gather*}
B\"{u}hring and Srivastava found expressions for the coef\/f\/icients $f_p(n|\al,\be)$, $h_p(n|\al,\be)$ as limits of hypergeometric polynomials \cite[(3.13)]{BS}. They computed this limit for
$p=3,4,5$ yielding \cite[(3.15), (3.17)]{BS}
\begin{gather}
h^s_3(n)\equiv h_3(n|1-\b+a_s,1-\a_{[s]}+a_s)\nonumber\\
\hphantom{h^s_3(n)}{} =-\frac{\Gamma(\psi_3)\Gamma(1-b_2+a_s+n)\Gamma(1-b_3+a_s+n)}{(1-\psi_3)_{n+1}n!\Gamma(\psi_3-b_2+a_s)\Gamma(\psi_3-b_3+a_s)}\,
{}_{3}F_{2} \left( \begin{matrix}\psi_3-1-n,b_1-\a_{[s]}\\ \psi_3-\b_{[1]}+a_s\end{matrix} \right),\!\!\!\!\label{eq:partial}\\
h^s_4(n)=-\frac{\Gamma(\psi_4)\Gamma(1-b_3+a_s+n)\Gamma(1-b_4+a_s+n)}{(1-\psi_4)_{n+1}n!\Gamma(\psi_4-b_3+a_s)\Gamma(\psi_4-b_4+a_s)}
\nonumber\\
\hphantom{h^s_4(n)=}{}
\times \sum\limits_{k=0}^{\infty}(\psi_4-1-n)_k \frac{(b_1+b_2-a_{i_1}-a_{i_2})_k(b_1+b_2-a_{i_1}-a_{i_3})_k}{(\psi_4-b_3+a_s)_k(\psi_4-b_4+a_s)_k}\nonumber\\
\hphantom{h^s_4(n)=}{}\times
{}_{3}F_{2} \left( \begin{matrix}-k,b_1-a_{i_1},b_2-a_{i_1}\\b_1+b_2-a_{i_1}-a_{i_2},b_1+b_2-a_{i_1}-a_{i_3}\end{matrix} \right),\nonumber
\end{gather}
where $\{a_{i_1},a_{i_2},a_{i_3}\}=\a_{[s]}$ (i.e., if $s=1$ then
$\{a_{i_1},a_{i_2},a_{i_3}\}=\{a_{2},a_{3},a_{4}\}$, if $s=2$ then
$\{a_{i_1},a_{i_2},a_{i_3}\}=\{a_{1},a_{3},a_{4}\}$, etc.). Expression for $p=5$ is quite cumbersome and can be found in~\cite[(3.19)]{BS}. In spite of their non-symmetric appearance, both formulas are invariant with respect to permutation of the elements of~$\b$.

Let us now cite the formulas for the coef\/f\/icients $g^s_p(n)$ for $p=2,3,4$. For $p=2$ the corresponding formulas can be read of\/f formula~(\ref{eq:2F1connection}).
For both $p=2$ and $p=3$ expressions for $g^s_p(n)$ have been found by N{\o}rlund, see~\cite[(2.10)]{Norlund}. They are
\begin{gather*}
g^{s}_{2}(n)=\frac{(b_1-a_{3-s})_{n}(b_2-a_{3-s})_{n}}{n!}, \qquad s=1,2,
\end{gather*}
and
\begin{gather*}
g^{s}_{3}(n)=\frac{(\psi_3-b_2+a_s)_n(\psi_3-b_3+a_s)_n}{n!}\,
{}_3F_2\left( \begin{matrix}-n,b_1-a_{i_1},b_1-a_{i_2}\\ \psi_3-b_2+a_s,\psi_3-b_3+a_s\end{matrix} \right), \qquad s=1,2,3,
\end{gather*}
where $\{a_{i_1},a_{i_2}\}=\a_{[s]}$. Notwithstanding the non-symmetric appearance, the last formula is symmetric with
respect to the elements of~$\b$. For $p=4$ we convert the expression for $f_p^s(n)$ calculated in \cite[(3.17)]{BS} into expressions for $g^s_p(n)$ using~(\ref{eq:Buehr-Norl-eq}) and the necessary renaming of parameters. This yields
\begin{gather*}
g_{4}^{s}(n)=\frac{(\psi_4-b_3+a_s)_n(\psi_4-b_4+a_s)_n}{n!}\\
\hphantom{g_{4}^{s}(n)=}{}\times
\sum\limits_{k=0}^{n}\frac{(-n)_k(b_1+b_2-a_{i_1}-a_{i_2})_k(b_1+b_2-a_{i_1}-a_{i_3})_k}{(\psi_4-b_3+a_s)_k(\psi_4-b_4+a_s)_k}
\\
\hphantom{g_{4}^{s}(n)=}{}\times
{}_{3}F_{2} \left( \begin{matrix}-k,b_1-a_{i_1},b_2-a_{i_1}\\b_1+b_2-a_{i_1}-a_{i_2},b_1+b_2-a_{i_1}-a_{i_3}\end{matrix} \right),
\end{gather*}
where $\{a_{i_1},a_{i_2},a_{i_3}\}=\a_{[s]}$. This formula is also invariant with respect to permutation of the elements of~$\b$.

\section{Main results}\label{section3}

Having made these preparations, we are ready to formulate our main results. First, we give explicit formulas for the N{\o}rlund's coef\/f\/icients~$g^s_p(1)$, $g^s_p(2)$, $g^s_p(3)$ with arbitrary $p$. Note that formula (\ref{eq:Norlund-explicit}) contains $p-2$ summations even for small $n$. Combined with (\ref{eq:Norl-xi}) and $g^{p}_{p}(0)=1$ our formulas essentially provide f\/irst four terms in asymptotic expansion of $G^{p,0}_{p,p}$ as $z\to1$. Then we derive a new way to compute the coef\/f\/icients $g^s_p(n)$ for all $n$ in terms of generalized Bernoulli polynomials in Theorem~\ref{th:norlundcoeff1}. Theorems~\ref{th:identitiesNorl} and~\ref{th:identitiesBuehr} contain new identities for the coef\/f\/icients $g^s_p(n)$, $h_p^s(n)$ and $D_n^{[k,s]}$ def\/ined by expansions~(\ref{eq:Norl-xi}),~(\ref{eq:G2pppNorlund}) and~(\ref{eq:Buehring13}), respectively. For small values of $n$ these identities lead to new and known relations involving hypergeometric and trigonometric functions, in particular, to identity~(\ref{eq:Ptol}).
\begin{Theorem}\label{th:g123}
The first four coefficients of the expansion \eqref{eq:Norl-xi} are given by
\begin{gather}\label{eq:gsp1}
g^{p}_{p}(0)=1,\qquad g^{p}_{p}(1)=\sum_{m=1}^{p-1}(b_{m+1}-a_m)\psi_m,\\
\label{eq:gsp2}
g^{p}_{p}(2)=\frac{1}{2}\sum_{m=1}^{p-1}(b_{m+1}-a_m)_2(\psi_m)_2+\sum_{k=2}^{p-1}(b_{k+1}-a_{k})(\psi_{k}+1)\sum_{m=1}^{k-1}(b_{m+1}-a_m)\psi_m
\end{gather}
and
\begin{gather}
g^{p}_{p}(3) = \frac{1}{6}\sum_{m=1}^{p-1}(b_{m+1}-a_m)_3(\psi_m)_3+\frac{1}{2}\sum_{k=2}^{p-1}(b_{k+1}-a_{k})(\psi_{k}+2)
\sum_{m=1}^{k-1}(b_{m+1}-a_m)_2(\psi_m)_2\nonumber\\
\hphantom{g^{p}_{p}(3) =}{}
+\frac{1}{2}\sum_{k=2}^{p-1}(\psi_{k}+1)_2
(b_{k+1}-a_{k})_2\sum_{m=1}^{k-1}(b_{m+1}-a_m)\psi_m\nonumber\\
\hphantom{g^{p}_{p}(3) =}{}
+\sum_{n=3}^{p-1}(b_{n+1}-a_n)(\psi_n+2)\sum_{k=2}^{n-1}(b_{k+1}-a_{k})(\psi_{k}+1)\sum_{m=1}^{k-1}(b_{m+1}-a_m)\psi_m,\label{eq:gsp3}
\end{gather}
where, as before, $\psi_m=\sum\limits_{j=1}^{m}(b_j-a_j)$.
\end{Theorem}

\begin{proof} We note that for $n=1$ summation in (\ref{eq:Norlund-explicit}) is over the index sets of the form
\begin{gather*}
\{j_0,j_1,\ldots_,j_{p-1}\}=\{0,\ldots,0,1,\ldots,1\},
\end{gather*}
where the number of ones changes from $1$ to $p-1$. In view of this observation rearrangement of the formula~(\ref{eq:Norlund-explicit}) yields~(\ref{eq:gsp1}).
Next, we prove~(\ref{eq:gsp3}). Summation in~(\ref{eq:Norlund-explicit}) is over all Young diagrams that f\/it $3\times(p-2)$ box. We break all possible diagrams in four disjoint groups as follows (by def\/inition $j_0=0$, $j_{p-1}=3$):
\begin{alignat*}{3}
&(1)\quad && j_0=j_1=\cdots={j_{m-1}}=0,\qquad j_m=\cdots=j_{p-1}=3,\qquad m\in\{1,2,\ldots,p-1\},&\\
&(2) \quad && j_0=j_1=\cdots={j_{m-1}}=0, \qquad j_m=\cdots=j_{k-1}=2,\qquad j_{k}=\cdots= j_{p-1}=3,& \\
&&& k\in\{2,\ldots,p-1\}, \qquad m\in\{1,\ldots,k-1\},\qquad m<k,&\\
&(3)\quad && j_0=j_1=\cdots={j_{m-1}}=0,\qquad j_m=\cdots=j_{k-1}=1,\qquad j_{k}=\cdots= j_{p-1}=3,&\\
&&& k\in\{2,\ldots,p-1\},\qquad m\in\{1,\ldots,k-1\},\qquad m<k,& \\
&(4) \quad && j_0=\cdots={j_{m-1}}=0, \qquad j_m=\cdots=j_{k-1}=1,\qquad j_{k}=\cdots=j_{n-1}=2,&\\
&&& j_{n}=\cdots=j_{p-1}=3,\qquad n\in\{3,\ldots,p-1\},\qquad k\in\{2,\ldots,n-1\},&\\
&&& m\in\{1,\ldots,k-1\}, \qquad m<k<n.&
\end{alignat*}
Summation over the f\/irst type of diagrams leads to the f\/irst term~(\ref{eq:gsp3}). Similarly, summation over the $i$-th type of diagrams leads to the $i$-th term~(\ref{eq:gsp3}) for $i=2,3,4$. Analogous considerations lead to~(\ref{eq:gsp2}).
\end{proof}

\begin{Remark}
Exchanging the roles of $a_p$ and $a_s$ formulas (\ref{eq:gsp1}), (\ref{eq:gsp2}) and (\ref{eq:gsp3}) lead to expressions for $g^s_p(1)$, $g^s_p(2)$, $g^s_p(3)$. Since each coef\/f\/icient $g^s_p(n)$ is a symmetric polynomial in the elements of~$\a_{[s]}$ and~$\b$, such expressions can be expanded in terms of some basis of symmetric polynomials. In particular, denoting by $e_k(x_1,\ldots,x_m)$ the $k$-th elementary symmetric polynomial of $x_1,\ldots,x_m$, we get
\begin{gather*}
g^s_p(1)=e_2(\b)-e_2(\a_{[s]})+e_1(\a_{[s]})(e_1(\a_{[s]})-e_1(\b))
\end{gather*}
and (with the help of \textit{Mathematica})
\begin{gather*}
g^s_p(2)=\frac{1}{2}e_1(\a_{[s]})^4-e_1(\a_{[s]})^3+\frac{1}{2}e_1(\a_{[s]})^2+\frac{3}{2}e_1(\a_{[s]})e_2(\a_{[s]})+\frac{1}{2}e_2(\a_{[s]})^2-\frac{1}{2}e_2(\a_{[s]})
\\
\hphantom{g^s_p(2)=}{}
-e_1(\a_{[s]})^2e_2(\a_{[s]})-\frac{1}{2}e_3(\a_{[s]})+\frac{1}{2}e_2(\b)^2+\frac{1}{2}e_2(\b)+\frac{1}{2}e_3(\b)-\frac{1}{2}e_1(\a_{[s]})e_1(\b)\\
\hphantom{g^s_p(2)=}{}
+\frac{3}{2}e_1(\a_{[s]})^2e_1(\b)-e_1(\a_{[s]})^3e_1(\b)
-\frac{1}{2}e_1(\a_{[s]})e_1(\b)^2+\frac{1}{2}e_1(\a_{[s]})^2e_1(\b)^2\\
\hphantom{g^s_p(2)=}{}
+e_1(\a_{[s]})^2e_2(\b)+\frac{1}{2}e_1(\b)e_2(\b)-e_1(\a_{[s]})e_1(\b)e_2(\b)-e_2(\a_{[s]})e_2(\b)-e_1(\b)e_2(\a_{[s]})\\
\hphantom{g^s_p(2)=}{}
+e_1(\a_{[s]})e_2(\a_{[s]})e_1(\b)-e_1(\a_{[s]})e_2(\b).
\end{gather*}
\end{Remark}

Next theorem gives a presumably new expression for N{\o}rlund's coef\/f\/icients $g_p^k(n)$ in terms of the generalized Bernoulli polynomials. Let us start by recalling that the Bernoulli--N{\o}rlund (or the gene\-ra\-li\-zed Bernoulli) polynomial $\Be^{(\sigma)}_{k}(x)$ is def\/ined by the generating function \cite[(1)]{Norlund61}:
\begin{gather*}
\frac{t^{\sigma}e^{xt}}{(e^t-1)^{\sigma}}=\sum\limits_{k=0}^{\infty}\Be^{(\sigma)}_{k}(x)\frac{t^k}{k!}.
\end{gather*}
In particular, $\Be^{(1)}_{k}(x)=\Be_{k}(x)$ is the classical Bernoulli polynomial.

\begin{Theorem}\label{th:norlundcoeff1}
Coefficients $g_p^k(n)$, defined in \eqref{eq:Norl-xi}, are given by any of the following formulas
\begin{gather}
g_p^k(n)=\sum\limits_{r=0}^{n}\frac{(-1)^{n-r}(r+1)_{n-r}}{(n-r)!}\tilde{l}_r\Be^{(n+1)}_{n-r}(2-a_k-\psi_p)\nonumber\\
\hphantom{g_p^k(n)}{}
=\sum\limits_{r=0}^n\frac{(-1)^{n-r} (\psi_p+r)_{n-r}}{(n-r)!}l_r\Be^{(n+\psi_p)}_{n-r}(1-a_k).\label{eq:Norlund-Bernoulli}
\end{gather}
Here $\tilde{l}_0=1$ and $\tilde{l}_r$, $r\ge1$, are found from the recurrence
\begin{gather}\label{eq:ñr}
\tilde{l}_r=\frac{1}{r}\sum\limits_{m=1}^r\widetilde{q}_m \tilde{l}_{r-m},
\end{gather}
where
\begin{gather*}
\widetilde{q}_m= \frac{(-1)^{m+1}}{m+1}
\left[\Be_{m+1}(a_k+\psi_p-1)-{\Be}_{m+1}(a_k)+\sum\limits_{j=1}^p{(\Be_{m+1}(a_j)}
-{\Be_{m+1}(b_j))}\right].
\end{gather*}
Similarly, the coefficients $l_r$ satisfy the recurrence relation
\begin{gather}\label{eq:lr}
l_r=\frac{1}{r}\sum\limits_{m=1}^r q_m l_{r-m}\qquad \text{with}\quad l_0=1
\end{gather}
and
\begin{gather*}
q_m=\frac{(-1)^{m+1}}{m+1} \sum\limits_{j=1}^p\big(\Be_{m+1}(a_j)-\Be_{m+1}(b_j)\big) .
\end{gather*}
\end{Theorem}

\begin{Remark}
It is known \cite[Lemma~1]{Kalinin} that the recurrences (\ref{eq:ñr}) and (\ref{eq:lr}) can be solved to give the following explicit expressions for~$l_r$:
\begin{gather*}
l_r=\sum\limits_{k_1+2k_2+\cdots+rk_r=r}\frac{q_1^{k_1}(q_2/2)^{k_2}\cdots (q_r/r)^{k_r}}{k_1!k_2!\cdots k_r!}
=\sum\limits_{n=1}^{r}\frac{1}{n!}\sum\limits_{k_1+k_2+\cdots+k_n=r}\prod\limits_{i=1}^{n}\frac{q_{k_i}}{k_i}.
\end{gather*}
Similar formula is of course true for $\tilde{l}_r$ once we write $\widetilde{q}_m$ instead of~$q_m$. Moreover, Nair~\cite[Section~8]{Nair} found a~determinantal expression for such solution which in our notation takes the form
\begin{gather*}
l_r=\frac{\det(\Omega_r)}{r!},\qquad \Omega_r=[\omega_{i,j}]_{i,j=1}^{r},\qquad \omega_{i,j} = \begin{cases}q_{i-j+1}(i-1)!/(j-1)!, & i\ge{j},\\
-1, & i=j-1,\\ 0, &i<j-1.
\end{cases}
\end{gather*}
\end{Remark}

\begin{proof} The theorem is a corollary of an expansion of the $H$-function of Fox found in our recent paper~\cite{KPJMAA2016}. Since Meijer's $G$-function is a particular case of Fox's~$H$-function, formula~(\ref{eq:Norlund-Bernoulli}) is a particular case of \cite[Theorem~1]{KPJMAA2016} once we set $p=q$, $\A=\B=(1,\ldots,1)$, $\nu=1$, $\mu=\psi_p$, $\theta=a_k-1$ in that theorem.
\end{proof}

As before, we use the shorthand notation
\begin{gather*}
\sin(\pi(\b-a_k))=\prod_{j=1}^{p}\sin(\pi(b_j-a_k)),\qquad \sin(\pi(\a_{[k]}-a_k))=\prod_{\substack{j=1\\j\ne{k}}}^{p}\sin(\pi(a_j-a_k)),
\end{gather*}
for the products and $[a]_j=a(a-1)\cdots(a-j+1)$ for the falling factorial.

\begin{Theorem}\label{th:identitiesNorl}
For each nonnegative integer $m$ the following identities holds
\begin{gather}\label{eq:identity1}
\sum\limits_{j=0}^{m}\frac{(-1)^{j}}{j!}\sum\limits_{k=1}^{p}[a_k]_jh^k_p(m-j)\frac{\sin(\pi(\b-a_k))}{\sin(\pi(\a_{[k]}-a_k))}=0,
\end{gather}
where the numbers $h^p_p(n)$ are defined by~\eqref{eq:Buehring13} and given explicitly by~\eqref{eq:hp-explicit}, $h^k_p(n)$ is obtained from $h^p_p(n)$ by exchanging the roles of~$a_p$ and~$a_k$; furthermore,
\begin{gather}\label{eq:identity2}
\sum\limits_{j=0}^{m}\frac{(-1)^{j}}{(\psi_p)_{m-j}j!}\left\{[a_s]_jg^s_p(m-j)\sin(\pi\psi_p)
-\sum\limits_{k=1}^{p}[a_k]_j g^k_p(m-j)\frac{\sin(\pi(\b-a_k))}{\sin(\pi(\a_{[k]}-a_k))}\right\}=0,
\end{gather}
where $s\in\{1,\ldots,p\}$ is chosen arbitrarily and the numbers
$g^k_p(n)$ are defined by expansion~\eqref{eq:Norl-xi} and
solve the recurrence~\eqref{eq:Norl-recur} in~$n$ and the
recurrence~\eqref{Norl-rec-p} in~$p$. They are given explicitly by~\eqref{eq:Norlund-explicit}.
\end{Theorem}

\begin{proof} Assume that the components of the vector $\a$ are distinct modulo~$1$. Substituting expansion~(\ref{eq:Buehring13}) into formula~(\ref{eq:Gp0pp}) and taking account of~(\ref{eq:Buehr-Norl-eq}) we obtain
\begin{gather*}
G^{p,0}_{p,p} \left( z\,\vline\,\begin{matrix}\b\\\a\end{matrix} \right)
=(1-z)^{\psi_p-1}\sum\limits_{k=1}^{p}z^{a_k}\frac{\Gamma(\a_{[k]}-a_k)\Gamma(1-\a_{[k]}+a_k)}{\Gamma(\b-a_k)\Gamma(1-\b+a_k)}
\sum_{n=0}^{\infty}\frac{\Gamma(1-\psi_p)}{(\psi_p)_n}g^k_p(n)(1-z)^n
\\
\hphantom{G^{p,0}_{p,p} \left( z\,\vline\,\begin{matrix}\b\\\a\end{matrix} \right)=}{}
+\sum\limits_{k=1}^{p}z^{a_k}\frac{\Gamma(\a_{[k]}-a_k)\Gamma(1-\a_{[k]}+a_k)}{\Gamma(\b-a_k)\Gamma(1-\b+a_k)}
\sum_{n=0}^{\infty}h_p^k(n)(1-z)^n
\\
\hphantom{G^{p,0}_{p,p} \left( z\,\vline\,\begin{matrix}\b\\\a\end{matrix} \right)}{}
=\frac{1}{\pi}(1-z)^{\psi_p-1}\Gamma(1-\psi_p)\sum_{n=0}^{\infty}\frac{(1-z)^n}{(\psi_p)_n}
\sum\limits_{k=1}^{p}z^{a_k}\frac{\sin(\pi(\b-a_k))}{\sin(\pi(\a_{[k]}-a_k))}g^k_p(n)
\\
\hphantom{G^{p,0}_{p,p} \left( z\,\vline\,\begin{matrix}\b\\\a\end{matrix} \right)=}{}
+\frac{1}{\pi}\sum_{n=0}^{\infty}(1-z)^n\sum\limits_{k=1}^{p}z^{a_k}\frac{\sin(\pi(\b-a_k))}{\sin(\pi(\a_{[k]}-a_k))}h_p^k(n),
\end{gather*}
where we applied Euler's ref\/lection formula
$\Gamma(z)\Gamma(1-z)=\pi/\sin(\pi{z})$. Further, substitute
N{\o}rlund's expansion~(\ref{eq:Norl-xi}) in place of $G$-function
on the left-hand side and rearrange terms to get
\begin{gather*}
(1-z)^{\psi_p-1}\sum\limits_{n=0}^{\infty}(1-z)^n\left\{\frac{z^{a_s}g^{s}_{p}(n)}{\Gamma(\psi_p+n)}
-\frac{\Gamma(1-\psi_p)}{(\psi_p)_n}\sum\limits_{k=1}^{p}\frac{z^{a_k}\sin(\pi(\b-a_k))}{\pi\sin(\pi(\a_{[k]}-a_k))}g^k_p(n)\right\}\\
\qquad{}=\sum_{n=0}^{\infty}(1-z)^n\sum\limits_{k=1}^{p}\frac{z^{a_k}\sin(\pi(\b-a_k))}{\pi\sin(\pi(\a_{[k]}-a_k))}h_p^k(n),
\end{gather*}
where $s\in\{1,\ldots,p\}$ can be chosen arbitrarily. Denote for brevity
\begin{gather*}
\gamma_{n,s}(z)=\frac{z^{a_s}g^{s}_{p}(n)}{\Gamma(\psi_p+n)}
-\frac{\Gamma(1-\psi_p)}{(\psi_p)_n}\sum\limits_{k=1}^{p}\frac{z^{a_k}\sin(\pi(\b-a_k))}{\pi\sin(\pi(\a_{[k]}-a_k))}g^k_p(n),\\
\chi_n(z)=\sum\limits_{k=1}^{p}\frac{z^{a_k}\sin(\pi(\b-a_k))}{\pi\sin(\pi(\a_{[k]}-a_k))}h_p^k(n),
\end{gather*}
so that the above equality reduces to
\begin{gather}\label{eq:gamma-chi}
(1-z)^{\psi_p-1}\sum\limits_{n=0}^{\infty}(1-z)^n\gamma_{n,s}(z)=\sum_{n=0}^{\infty}(1-z)^n\chi_n(z).
\end{gather}
Assume for a moment that $\psi_p$ is not an integer. We know that both series in~(\ref{eq:gamma-chi}) converge in $|z-1|<1$.
Furthermore, all functions $\gamma_{n,s}(z)$ and $\chi_n(z)$ are analytic the same disk. This implies that~(\ref{eq:gamma-chi}) is only possible for all $z$ in a~disk centered at~$1$ if
\begin{gather}\label{eq:zero}
\sum\limits_{n=0}^{\infty}(1-z)^n\gamma_{n,s}(z)\equiv0\qquad \text{and}\qquad \sum_{n=0}^{\infty}(1-z)^n\chi_n(z)\equiv0.
\end{gather}
In terms of the functions $\chi_n(z)$ the claimed identity (\ref{eq:identity1}) takes the form
\begin{gather}\label{eq:chi-identity}
\sum\limits_{j=0}^{m}\frac{(-1)^{j}}{j!}\chi_{m-j}^{(j)}(1)=0,
\end{gather}
which we prove by induction in~$m$.

Letting $z\to1$ in the second identity in (\ref{eq:zero}) we get $\chi_0(1)=0$ which
establishes our claim for $m=0$. Next, suppose~(\ref{eq:chi-identity}) holds for $m=0,1,2,\ldots,r-1$. Divide the second identity in~(\ref{eq:zero}) by $(1-z)^r$ and expand each $\chi_n(z)$, $n=0,1,\ldots,r$, in Taylor series around $z=1$:
\begin{gather*}
0=\sum_{n=0}^{r}(1-z)^{n-r}\bigl\{\chi_n(1)+\chi_n'(1)(z-1)+\cdots\\
\hphantom{0=}{}+\chi_n^{(r-n)}(z-1)^{r-n}/(r-n)!+O\big((z-1)^{r-n+1}\big)\bigr\}+\sum_{n=r+1}^{\infty}(1-z)^{n-r}\chi_n(z)\\
\hphantom{0}{}
=\sum\limits_{k=0}^{r}(1-z)^{-k}\sum\limits_{i=0}^{r-k}\frac{(-1)^i}{i!}\chi_{r-k-i}^{(i)}(1)+O(z-1)
=\sum\limits_{i=0}^{r}\frac{(-1)^i}{i!}\chi_{r-k-i}^{(i)}(1)+O(z-1),
\end{gather*}
where the last equality is by induction hypothesis. We now obtain~(\ref{eq:chi-identity}) for $m=r$ on letting $z\to1$ in this formula. It is immediate to check that the claimed identity~(\ref{eq:identity2}) takes the form{\samepage
\begin{gather*}%\label{eq:gamma-identity}
\sum\limits_{j=0}^{m}\frac{(-1)^{j}}{j!}\gamma_{m-j,s}^{(j)}(1)=0,
\end{gather*}
which can be demonstrated in a similar fashion starting with the f\/irst identity in~(\ref{eq:zero}).}

Finally, we remove the assumption that $\psi_p$ is not an integer. The left-hand sides of (\ref{eq:identity1}) and (\ref{eq:identity2}) are analytic functions of, say, parameter $b_1$ except for possible poles. Identities (\ref{eq:identity1}) and (\ref{eq:identity2}) for non-integer $\psi_p$ then clearly imply by analytic continuation that these poles are removable and both identities hold for all $\psi_p$.
\end{proof}

\begin{Corollary}\label{cr:3F2circular}
Suppose for some $i\in\{1,2,3\}$ inequality $\Re(b_i)<\Re(a_k+1)$ holds for $k\in\{1,2,3\}$. Then
\begin{gather}\label{eq:3F2circular1}
\sum\limits_{k=1}^{3}\frac{\sin(\pi(\b-a_k))\Gamma(1-\b_{[i]}+a_k)}{\sin(\pi(\a_{[k]}-a_k))\Gamma(\psi-\b_{[i]}+a_k)}\,
{}_{3}F_{2} \left( \begin{matrix}\psi-1,b_i-\a_{[k]}\\\psi-\b_{[i]}+a_k\end{matrix} \right)=0
\end{gather}
and
\begin{gather}
\sum\limits_{k=1}^{3}
\frac{\sin(\pi(\b-a_k))\Gamma(1-\b_{[i]}+a_k)}{\sin(\pi(\a_{[k]}-a_k))\Gamma(\psi-\b_{[i]}+a_k)}
\left\{(1-\b_{[i]}+a_k)\,
{}_{3}F_{2} \left( \begin{matrix}\psi-2,b_i-\a_{[k]}\\\psi-\b_{[i]}+a_k\end{matrix} \right)\right.\nonumber\\
\left.\qquad{}
-a_k(2-\psi)\, {}_{3}F_{2} \left( \begin{matrix}\psi-1,b_i-\a_{[k]}\\ \psi-\b_{[i]}+a_k\end{matrix} \right)\right\}=0,\label{eq:3F2circular2}
\end{gather}
where $\psi=\sum\limits_{k=1}^{3}(b_k-a_k)$ and ${}_pF_{p-1}$ without argument is understood as ${}_pF_{p-1}(1)$.
\end{Corollary}

\begin{proof}
Put $p=3$. Identities (\ref{eq:3F2circular1}) and (\ref{eq:3F2circular2}) are now reformulations of~(\ref{eq:identity1}) for $m=0$ and $m=1$, respectively. The coef\/f\/icients~$h_3^k(0)$ and~$h_3^k(1)$ have been computed by for\-mu\-la~(\ref{eq:partial}).
\end{proof}

The next corollary is a rewriting of (\ref{eq:identity2}) for $m=0$ in view of $g^k_p(0)=1$.

\begin{Corollary}\label{cr:ptolemy}
For any complex vectors $\a$, $\b$ the following identity holds
\begin{gather}\label{eq:ptolemy}
\sum\limits_{k=1}^{p}\frac{\sin(\pi(\b-a_k))}{\sin(\pi(\a_{[k]}-a_k))} =\sin(\pi\psi_p).
\end{gather}
The right-hand side gives a continuous extension of the left-hand side if $\a_{[k]}-a_k$ contains integers.
\end{Corollary}

\begin{Remark}
For $p=2$ this is equivalent to Ptolemy's theorem: if a quadrilateral is inscribed in a circle then the product of the lengths of its diagonals is equal to the sum of the products of the lengths of the pairs of opposite sides, which can be written as
\begin{gather*}
\sin(\theta_3-\theta_1)\sin(\theta_4-\theta_2)=\sin(\theta_2-\theta_1)\sin(\theta_4-\theta_3)+\sin(\theta_4-\theta_1)\sin(\theta_3-\theta_2).
\end{gather*}
Further details regarding the history behind the identity~(\ref{eq:ptolemy}) can be found in the introduction and~\cite{Johnson}.
\end{Remark}

\begin{Corollary}\label{cr:symmetric}
For each $m\in\N_0$ and each $p\in\N$ the function
\begin{gather}
F_{p,m}(\a,\b)=\sum\limits_{j=0}^{m}\frac{(-1)^j}{j!}[a_k]_j[\psi_p+m-1]_jg^k_p(m-j)\nonumber\\
\hphantom{F_{p,m}(\a,\b)}{}
=\sum\limits_{j=0}^{m}\sum\limits_{r=0}^{m-j}
\frac{(-1)^{m-j}(\psi_p+r)_{m-r}}{j!(m-j-r)!}l_r[a_k]_j\Be^{(m-j+\psi_p)}_{m-j-r}(1-a_k)\label{eq:Fsymmetry}
\end{gather}
is independent of $k$ and represents a symmetric polynomial in the
components of~$\a$ and~$\b$ $($separately$)$. Here $l_r$ is defined by the recurrence relation~\eqref{eq:lr}.
\end{Corollary}

\begin{proof}
Indeed, the f\/irst formula in (\ref{eq:Fsymmetry}) follows from (\ref{eq:identity2}) once we open the braces and apply the obvious relation $1/(\psi_p)_{m-j}=[\psi_p+m-1]_{j}/(\psi_p)_m$. Substitution of (\ref{eq:Norlund-Bernoulli})
for $g^k_p$ leads to the second formula.
\end{proof}

{\sloppy The following theorem can be viewed as a new method for computing the coef\/f\/icients $h_p(n|\al,\be)$ in expansion~(\ref{eq:Buehring1}) given by the multiple sum~(\ref{eq:hp-explicit}) by relating them to the numbers $D_n^{[k,s]}$ given by the single sums~(\ref{eq:Norlund5.35}) and ~(\ref{eq:Norlund5.36}).

}

\begin{Theorem}\label{th:identitiesBuehr}
For each nonnegative integer $n$ and arbitrary $s\in\{1,\ldots,p\}$ the following identity holds
\begin{gather}\label{eq:G2ppp}
h_p^s(n)=-\frac{1}{\pi\sin(\pi\psi_p)}\sum_{\substack{k=1\\k\ne{s}}}^{p} \frac{\sin(\pi(\b-a_k))}{\sin(\pi(\a_{[k,s]}-a_k))}D_n^{[k,s]},
\end{gather}
where $h_p^s(n)$ is defined by expansion \eqref{eq:Buehring13} and given explicitly by~\eqref{eq:hp-explicit}, while $D_n^{[k,s]}$ are given by~\eqref{eq:Norlund5.35} or~\eqref{eq:Norlund5.36}. Moreover, for arbitrary distinct integers $s$, $i$, $k$ from the set $\{1,2,\ldots,p\}$ the following identity holds
\begin{gather}
\sum\limits_{j=0}^{n}\frac{(-1)^{j}}{j!}\left([a_i]_j\sin(\pi(a_s-a_i))D_{n-j}^{[s,i]}
+[a_k]_j\sin(\pi(a_i-a_k))D_{n-j}^{[i,k]}\right.\nonumber\\
\left.
\hphantom{\sum\limits_{j=0}^{n}\frac{(-1)^{j}}{j!}}{} +[a_s]_j\sin(\pi(a_k-a_s))D_{n-j}^{[k,s]}\right)=0.\label{eq:G2pppNorlund5.3.1}
\end{gather}
\end{Theorem}
\begin{proof}
To prove~(\ref{eq:G2ppp}) it suf\/f\/ices to substitute expansion~(\ref{eq:G2pppNorlund}) into formula
(\ref{eq:regularconnection}) and equate coef\/f\/icients. Identity~(\ref{eq:G2pppNorlund5.3.1}) is a direct consequence of~(\ref{eq:G2pppNorlund5.3}).
\end{proof}

\begin{Corollary}
For each $n\in\N_0$ the following identity holds true
\begin{gather*}
\frac{\Gamma(a_3-a_1)}{\Gamma(\b-a_1)\Gamma(2+a_1+a_2-b_1-b_2+n)}\,
{}_{3}F_{2} \left( \begin{matrix}1+a_1-b_1,1+a_1-b_2,b_3-a_3\\1+a_1-a_3,2+a_1+a_2-b_1-b_2+n\end{matrix} \right)
\\
\quad{}
+\frac{\Gamma(a_1-a_3)}{\Gamma(\b-a_3)\Gamma(2+a_2+a_3-b_1-b_2+n)}\,
{}_{3}F_{2} \left( \begin{matrix}1+a_3-b_1,1+a_3-b_2,b_3-a_1\\1+a_3-a_1,2+a_2+a_3-b_1-b_2+n\end{matrix} \right)
\\
{} =\frac{1}{\Gamma(2-\psi+n)\Gamma(\psi+a_2-b_1)\Gamma(\psi+a_2-b_2)}\,
{}_{3}F_{2} \left( \begin{matrix}b_3-a_1,b_3-a_3,\psi-1-n\\\psi+a_2-b_1,\psi+a_2-b_2\end{matrix} \right).
\end{gather*}
\end{Corollary}

\begin{proof} For $p=3$, $s=2$ formula (\ref{eq:G2ppp}) takes the form
\begin{gather}\label{eq:G2ppp1}
h_3^2(n)=-\frac{1}{\pi\sin(\pi\psi)}\left(\frac{\sin(\pi(\b-a_1))}{\sin(\pi(a_3-a_1)}D_n^{[1,2]}+
\frac{\sin(\pi(\b-a_3))}{\sin(\pi(a_1-a_3)}D_n^{[3,2]}\right).
\end{gather}
Now, write $h_3^2(n)$ according to~(\ref{eq:partial}) and exchange the roles of~$b_1$ and~$b_3$. Next, apply
Chu--Vandermonde identity on the right-hand side of~(\ref{eq:Norlund5.35}) to get
\begin{gather*}
D_n^{[k,s]}=\frac{\Gamma(1-\b+a_k)\Gamma(1-b_1+a_s+n)\Gamma(1-b_2+a_s+n)}{\Gamma(1-\a_{[k,s]}+a_k)\Gamma(2+a_k+a_s-b_1-b_2+n)n!}\\
\hphantom{D_n^{[k,s]}=}{}
\times {}_3F_{2}\left( \begin{matrix}1-b_1+a_k,1-b_2+a_k, b_3-\a_{[k,s]}
\\
2+a_k+a_s-b_1-b_2+n, 1-\a_{[k,s]}+a_k \end{matrix} \right).
\end{gather*}
Substituting this into (\ref{eq:G2ppp1}), applying Euler's ref\/lection formula for the gamma function and rearranging we get the claimed identity.
\end{proof}

\appendix

\section[Def\/inition of Meijer's $G$-function revisited]{Def\/inition of Meijer's $\boldsymbol{G}$-function revisited}\label{appendixA}

Meijer's $G$-function has been def\/ined in the introduction, where we mentioned various aspects that need to
be clarif\/ied in order that this def\/inition be consistent. Most accurate information with proofs regarding
 $G$-function's def\/inition is contained, in our opinion, in the series of papers of Meijer himself~\cite{Meijer},
 the paper~\cite{Braaksma} by Braaksma and in the f\/irst chapters of the books~\cite{ParKam} and~\cite{KilSaig}.
 Further facts are scattered in the literature with most comprehensive collection being \cite[Chapter~8]{PBM3},~\cite{NIST} and especially~\cite{Wolfram}. An accessible introduction to $G$-function can be found in a~nice recent
 survey by Beals and Szmigielski~\cite{BealsSzmig}. In this paper we only deal with the function~ $G^{m,n}_{p,p}$.
 For convenience, we have gathered all the necessary information regarding its def\/inition in the following theorem.

\begin{Theorem}\label{th:sumres}
Denote $a^*=m+n-p$, $\psi=\sum\limits_{k=1}^{p}(b_k-a_k)$ and
\begin{gather*}
\G(s)=\frac{\Gamma(b_1 + s)\cdots\Gamma(b_m + s)\Gamma(1-a_1 - s)\cdots\Gamma(1-a_n - s)}
{\Gamma(a_{n+1} + s)\cdots\Gamma(a_p + s)\Gamma(1-b_{m+1} - s)\cdots\Gamma(1-b_{p} - s)}.
\end{gather*}
\begin{enumerate}\itemsep=0pt
\item[$(a)$] If $|z|<1$ then the integral in~\eqref{eq:G-defined} converges for $\L=\L_{-}$ and
\begin{gather}\label{eq:sumresleft}
G^{m,n}_{p,p}\left(z\,\vline\, \begin{matrix} \a\\
\b\end{matrix} \right)=\sum\limits_{j=1}^m\sum\limits_{l=0}^\infty\res_{s=b_{jl}}\G(s)z^{-s}, \qquad b_{jl}=-b_j-l.
\end{gather}
If, in addition, $a^*>0$ and $|\arg(z)|<a^*\pi$ or $a^*=0$,
$\Re(\psi)<0$ and $0<z<1$ then the integral in~\eqref{eq:G-defined} also converges for $\L=\L_{i\gamma}$
and has the same value.

\item[$(b)$] If $|z|>1$ then the integral in~\eqref{eq:G-defined} converges for $\L=\L_{+}$ and
\begin{gather*}%\label{eq:sumresright}
G^{m,n}_{p,p}\left(z\,\vline\, \begin{matrix} \a\\
\b\end{matrix} \right)=-\sum\limits_{i=1}^n\sum\limits_{k=0}^\infty\res_{s=a_{ik}}\G(s)z^{-s}, \qquad a_{ik}=1-a_i+k.
\end{gather*}
If, in addition, $a^*>0$ and $|\arg(z)|<a^*\pi$ or $a^*=0$,
$\Re(\psi)<0$ and $z>1$ then the integral in~\eqref{eq:G-defined} also converges for~$\L=\L_{i\gamma}$
and has the same value.
\end{enumerate}
\end{Theorem}

\begin{Remark}
If $|z|=1$ and $\Re(\psi)<-1$ then according to
\cite[Theorem~1.1]{KilSaig} both integrals over $\L_{-}$ and over
$\L_{+}$ exist. If, in addition, $a^*>0$ and $|\arg(z)|<a^*\pi$,
then the integral over~$\L_{i\gamma}$ also exists. We found no
proof in the literature that these integrals are equal.
\end{Remark}

\begin{proof} The claims regarding the contours $\L_{+}$ and $\L_{-}$ have been demonstrated in \cite[Theorems~1.1, 1.2]{KilSaig}. They follow in a relatively straightforward manner from Stirling's asymptotic formula for the gamma function and have been observed by Meijer himself in \cite[Section~1]{Meijer}. If $a^*=0$, $\Re\psi<0$ the result was proved by Kilbas and Saigo \cite[Theorem~3.3]{KilSaig}. For $a^*>0$ and $|z|<1$ \cite[Theorem~D]{Meijer} states that the
integrals over $\L_{-}$ and $\L_{i\gamma}$ coincide for any real $\gamma$ which covers case (a) of the theorem. For the proof
Meijer refers to \cite[Section~14.51]{WW}, where it is essentially demonstrated that
\begin{gather}\label{eq:intzero}
\lim_{R\to\infty}\int_{l^-_\gamma(R)}\frac{\prod\limits_{j=1}^{m}\Gamma(s+b_j)\prod\limits_{i=1}^{n}\Gamma(1-s-a_i)}
{\prod\limits_{i=n+1}^{p}\Gamma(s+a_i)\prod\limits_{j=m+1}^{p}\Gamma(1-s-b_j)}z^{-s}ds=0,
\end{gather}
where $l^-_\gamma(R)$ is the shortest arc of the circle $|s|=R$ which connects the contours $\L_{-}$ and $\L_{i\gamma}$ (in fact, Whittaker and Watson proved a particular case, but the proof for the general case goes along exactly the same lines). Since~$\G(s)$ has no poles in the domain bounded by $\L_{-}$ and $\L_{i\gamma}$ this leads to equality of the integrals over $\L_{-}$ and $\L_{i\gamma}$ stated by Meijer. It remains to consider the case $|z|>1$, $\L=\L_{i\gamma}$. Denote by $l^+_\gamma(R)$ the ref\/lection of $l^{-}_\gamma(R)$ with respect to $\L_{i\gamma}$. For $|z|>1$ we get by changing $s$ to $-s$
\begin{gather*}
\int_{l^{+}_\gamma(R)} \frac{\prod\limits_{j=1}^{m}\Gamma(s+b_j)\prod\limits_{i=1}^{n}\Gamma(1-s-a_i)z^{-s}}
{\prod\limits_{i=n+1}^{p}\Gamma(s+a_i)\prod\limits_{j=m+1}^{p}\Gamma(1-s-b_j)}ds\\
\qquad{} = - \int_{l^{-}_{\gamma}(R)} \frac{\prod\limits_{j=1}^{m}\Gamma(-s+b_j)\prod\limits_{i=1}^{n}\Gamma(1+s-a_i)z^{s}}
{\prod\limits_{i=n+1}^{p}\Gamma(-s+a_i)\prod\limits_{j=m+1}^{p}\Gamma(1+s-b_j)}ds\\
\qquad{} =- \int_{l^{-}_{\gamma}(R)} \frac{\prod\limits_{i=1}^{n}\Gamma(s+b_i')\prod\limits_{j=1}^{m}\Gamma(1-s-a_j')}
{\prod\limits_{j=m+1}^{p}\Gamma(s+a_j')\prod\limits_{i=n+1}^{p}\Gamma(1-s-b_i')} (1/z)^{-s}ds.
\end{gather*}
Here $b_i'=1-a_i$, $a_j'=1-b_j$. In view of (\ref{eq:intzero}), we immediately conclude that for $|z|>1$
\begin{gather*}
\lim_{R\to\infty}\int_{l^+_{\gamma}(R)}\frac{\prod\limits_{j=1}^{m}\Gamma(s+b_j)\prod\limits_{i=1}^{n}\Gamma(1-s-a_i)}
{\prod\limits_{i=n+1}^{p}\Gamma(s+a_i)\prod\limits_{j=m+1}^{p}\Gamma(1-s-b_j)}z^{-s}ds=0,
\end{gather*}
which implies that the integrals over $\L_{+}$ and $\L_{i\gamma}$ coincide.
\end{proof}

\begin{Remark} If $p>q$ ($q>p$) the integral in (\ref{eq:G-defined}) exists for $\L=\L_{+}$ ($\L=\L_{-}$) and all
complex $z\ne{0}$ and is equal to the corresponding sum of residues \cite[Theorems~1.1 and 1.2]{KilSaig}. At the same time
if $a^*>0$ and $|\arg(z)|<a^*\pi$ or $a^*=0$ and $z>0$, $z\ne{1}$, the integral in (\ref{eq:G-defined}) also exists for
$\L=\L_{i\gamma}$. Most authors assume in this situation that $\L_{i\gamma}$ can be deformed into $\L=\L_{+}$ if $p>q$ or
$\L=\L_{-}$ if $q>p$ without altering the value of the integral. However, we were unable to f\/ind any proof of this claim in the
literature.
\end{Remark}

\begin{Remark}
 It follows from the above theorem that $G^{m,n}_{p,p}(z)$ is analytic in the sector $|\arg(z)|$ $<a^*\pi$
if $a^*>0$ (since the integral converges uniformly in~$z$ for $\L=\L_{i\gamma}$), while for $a^*\leq0$ we get two dif\/ferent
analytic functions~-- one def\/ined inside and the other outside of the unit circle, see~\cite[(8.2.2.7)]{PBM3}.
\end{Remark}

In the proof of the above theorem we essentially used the next well-known ref\/lection property of $G$-function
\begin{gather*}
G^{m,n}_{p,q} \left(\frac{1}{z}\,\vline\,\begin{matrix}\a\\ \b\end{matrix} \right)
=G^{n,m}_{q,p} \left(z\,\vline\,\begin{matrix}1-\b\\ 1-\a\end{matrix} \right).
\end{gather*}
It is important to note that by Theorem~\ref{th:sumres}(b)
\begin{gather*}
G^{p,0}_{p,p} \left( z\, \vline\, \begin{matrix}\b\\ \a\end{matrix} \right)=0 \qquad \text{for} \quad |z|>1,
\end{gather*}
which is, of course, dif\/ferent from the analytic continuation of the right-hand side of~(\ref{eq:Gp0pp}). The Mellin transform of $G^{p,0}_{p,p}$ exists if either $\Re(\psi)>0$ or $\psi=-m\in\N_0$. In the former case \cite[Theorem~2.2]{KilSaig}
\begin{gather}\label{eq:GMellin}
\int_{0}^{\infty}x^{s-1}G^{p,0}_{p,p} \left( x\,\vline\,\begin{matrix}\b\\ \a\end{matrix} \right)dx
=\int_{0}^{1}x^{s-1}G^{p,0}_{p,p} \left( x\,\vline\,\begin{matrix}\b\\ \a\end{matrix} \right)dx
=\frac{\Gamma(\a+s)}{\Gamma(\b+s)}
\end{gather}
for $\Re(s)>-\Re(\a)$. If $\psi=-m\in\N_0$ then \cite[(2.28)]{Norlund}
\begin{gather*}%\label{eq:GMellinNorlund}
\int_{0}^{\infty}x^{s-1}G^{p,0}_{p,p} \left( x\,\vline\,\begin{matrix}\b\\ \a\end{matrix} \right)dx
=\int_{0}^{1}x^{s-1}G^{p,0}_{p,p} \left( x\,\vline\,\begin{matrix}\b\\\a\end{matrix} \right)dx
=\frac{\Gamma(\a+s)}{\Gamma(\b+s)}-q(s)
\end{gather*}
for $\Re(s)>-\Re(\a)$. Here $q(s)$ is a polynomial of degree $m$ given
by
\begin{gather*}
q(s)=\sum\limits_{j=0}^{m}g^k_{p}(m-j)(s+a_k-j)_j,\qquad k=1,2,\ldots,p,
\end{gather*}
where the coef\/f\/icients $g^k_{p}(n)$ are def\/ined in expansion (\ref{eq:Norl-xi}) and are given explicitly by~(\ref{eq:Norlund-explicit}). Note that $g^k_{p}(n)$ depends on $k$ while the polynomial~$q(s)$ is the same for each $k\in\{1,\ldots,p\}$.

\subsection*{Acknowledgements}
This work has been supported by the Russian Science Foundation under project 14-11-00022.
We are also indebted to anonymous referees for a number of useful remarks that helped to improve the exposition.

\pdfbookmark[1]{References}{ref}
\LastPageEnding

\end{document}